\documentclass[12pt]{amsart}%reqno,
\DeclareFontFamily{OML}{script}{}
\DeclareFontShape{OML}{script}{m}{it}
{ <5-20> rsfs10 }{}
\DeclareMathAlphabet{\mathscript}{OML}{script}{m}{it}

\usepackage{color}
\ifx\red\undefined
\newcommand{\red}{\color{red}}

\fi
\usepackage[ansinew]{inputenc}
\usepackage[encapsulated]{CJK}

\usepackage{graphicx,graphics}
\usepackage{cite}
\usepackage{pifont}
\usepackage{subfigure}
\usepackage{amscd}
\usepackage{latexsym}
\usepackage{color}
\usepackage{multicol}
\usepackage{amsmath,amssymb,amsthm,amsfonts,mathrsfs,amsxtra,indentfirst,mathtools,appendix}
\usepackage{hyperref}
\hypersetup{hypertex=true,
            colorlinks=true,
            linkcolor=blue,
            anchorcolor=blue,
            citecolor=blue}
\usepackage{bm}
\allowdisplaybreaks[4]

\newcommand{\Rmnum}[1]{\uppercase\expandafter{\romannumeral #1}}

\ifx\text\undefined
\newcommand{\text}{\mbox}
\fi
\ifx\operatorname\undefined
\newcommand{\operatorname}{\mathop}
\fi
\allowdisplaybreaks[3]

\newtheorem{theorem}{Theorem}[section]
\newtheorem{lemma}[theorem]{Lemma}
\newtheorem{proposition}[theorem]{Proposition}
\newtheorem{remark}[theorem]{Remark}

\theoremstyle{remark}

\newcommand{\void}[1]{}

\numberwithin{equation}{section}

\renewcommand\d{\mathrm{d}}
\renewcommand{\leq}{\leqslant}
\renewcommand{\geq}{\geqslant}

\DeclareMathOperator{\divergence}{div}
\DeclareMathOperator{\curl}{curl}

\begin{document}
\begin{CJK}{UTF8}{gkai}

\title[1]{On Type I blowup and $\varepsilon$-regularity criteria for suitable weak solutions to the 3D incompressible MHD equations}
\author{Wentao Hu and Zhengce Zhang}
\date{}
\address[Wentao Hu]{School of Mathematics and Statistics, Xi'an Jiaotong University,
Xi'an, 710049, P. R. China}
\email{huwentao@stu.xjtu.edu.cn}
\address[Zhengce Zhang]{School of Mathematics and Statistics, Xi'an Jiaotong University,
Xi'an, 710049, P. R. China}
\email{zhangzc@mail.xjtu.edu.cn}
\thanks{Corresponding author: Zhengce Zhang}

\begin{abstract}
We study interior $\varepsilon$-regularity and Type I blowup criteria for suitable weak solutions to the three-dimensional incompressible MHD equations. Our starting point is a direct iteration scheme for the classical Caffarelli--Kohn--Nirenberg scaled energy quantities $A,E,C$ and $D$, which yields $\varepsilon$-regularity criteria under smallness assumptions on the velocity field $u$ and boundedness assumptions on the magnetic field $b$, with the underlying scaling-invariant quantities chosen independently. As an intermediate step, we prove that finiteness of one such scaling-invariant quantity for each of $u$ and $b$ allows only Type I blowup, in the sense that $A(u,b;r)+E(u,b;r)+C(u,b;r)+D(p;r)<\infty$ for small $r$. This extends Seregin's Type I criteria for the Navier--Stokes equations to the MHD setting and provides a natural point of departure for the analysis of Type II blowup. By interpolation and embedding, we further obtain $\varepsilon$-regularity criteria and Type I characterisations in terms of general scaled mixed Lebesgue norms for $u$ and $b$, with independent exponent choices. While we do not aim to sharpen existing mixed-norm $\varepsilon$-regularity criteria, the present formulation offers a unified and comparatively direct route that is naturally compatible with the Type I framework; in particular, the mixed-norm Type I description does not follow from earlier mixed-norm $\varepsilon$-regularity proofs by a formal replacement of the smallness parameter.
\end{abstract}

\maketitle

\section{Introduction}
Consider the three-dimensional incompressible magnetohydrodynamic (MHD) equations:
\begin{equation}
    \label{eq: MHD}
    \begin{dcases}
    \partial_t u+u \cdot \nabla u-\Delta u+\nabla p=b \cdot \nabla b, \\
    \partial_t b+u \cdot \nabla b-\Delta b=b \cdot \nabla u, \\
    \divergence u=\divergence b=0,
    \end{dcases}
\end{equation}
where the unknown vector fields $u,b$ and scalar field $p$ represent the velocity field, the magnetic field, and the pressure, respectively. The system \eqref{eq: MHD} depicts the motion of viscous incompressible electrically conducting fluids in the absence of external forces. When $b \equiv 0$, the system \eqref{eq: MHD} reduces to the three-dimensional incompressible Navier--Stokes equations
\begin{equation}
    \label{eq: NS}
    \begin{dcases}
    \partial_t u+u \cdot \nabla u-\Delta u+\nabla p=0, \\
    \divergence u=0,
    \end{dcases}
\end{equation}
which has been studied intensively during the past decades, see \cite{leray1934Mouvement,hopf1950Uber,prodi1959Teorema,serrin1962Interior,serrin1963Initial,ladyzhenskaya1967Uniqueness,caffarelli1982Partial,lin1998New,escauriaza2003Solutions,CSYT2008,CSYT2009,koch2009liouville} and references therein. To be specific, Leray \cite{leray1934Mouvement} and Hopf \cite{hopf1950Uber} proved the existence of weak solutions to \eqref{eq: NS} in $\mathbb{R}^3$ and bounded domains in $\mathbb{R}^3$, respectively. Ladyzhenskaya, Prodi, and Serrin \cite{prodi1959Teorema,serrin1962Interior,serrin1963Initial,ladyzhenskaya1967Uniqueness} studied the regularity of solutions in the class $L_t^s L_x^q$ with $2/s+3/q \leq 1$, $s \geq 2$, $q>3$ independently. This type of conditions (LPS conditions for short), roughly speaking, enables us to estimate the nonlinear term $u \cdot \nabla u$ like a linear term, which leads to regularity. However, the borderline case $s=\infty$, $q=3$ is quite different and much more difficult. It was not until 2003 that Escauriaza, Seregin and {\v S}ver\'ak \cite{escauriaza2003Solutions} proved the regularity in this case by developing a new method based on the unique continuation theory. See also \cite{seregin2012Certain} for further improvements. Similar results were also obtained for MHD equations. For instance, Duvaut and Lions \cite{duvaut1972Inequations} proved the global existence of weak solutions to \eqref{eq: MHD} in simply connected bounded domains, whereas Sermange and Temam \cite{sermange1983Mathematical} studied the case where $u,b$ are periodic in space variables. Regarding the regularity under the LPS conditions, Wu \cite{wu2004Regularity} proved analogous results to Navier--Stokes equations, with both $u,b$ being assumed to belong to $L_t^s L_x^q(\mathbb{R}^3)$, where $2/s+3/q=1$, $s \geq 2$, $q>3$, while He and Xin \cite{he2005Regularity} showed that the assumption on $b$ can be dropped. See also \cite{zhou2005Remarks} for the case where $u \in L_t^s L_x^q$ with $2/s+3/q \leq 1$, $s \geq 2$, $q>3$. Compared with \cite{wu2004Regularity}, these results only made assumptions on $u$, which suggests that the velocity field plays a more dominant role than the magnetic field in the regularity theory, just as the numerical results in \cite{politano1995Current} implied. The borderline case $s=\infty$, $q=3$ is quite different, and by applying similar technique to \cite{escauriaza2003Solutions}, Mahalov, Nicolaenko and Shilkin \cite{mahalov2006Solutions} proved the regularity of solutions in $L_t^\infty L_x^3$. There are also many papers focusing on imposing mixed type of LPS conditions on components of $(u,b)$. For example, Ji and Lee \cite{ji2010Regularity} considered conditions on planar components $(u_h,b_h)$ of $(u,b)$ or conditions on $u_h$ and $b_3$. Jia and Zhou \cite{jia2012Regularity} studied conditions on $u_3$, $b$ and $\partial_3 u_h$. Other results of this type can be found, e.g., in \cite{jia2015New,jia2016Regularity}.

In 1982, Caffarelli, Kohn and Nirenberg \cite{caffarelli1982Partial} introduced the notion of suitable weak solutions of 3D incompressible Navier--Stokes equations, which were defined as weak solutions $(u,p) \in (L_t^\infty L_x^2 \cap L_t^2 H^1) \times L^{3/2}$ that satisfy the local energy inequality. By showing the regularity of suitable weak solutions at any space-time point under the smallness assumptions of certain scaled energy quantities near that point (these are known as the $\varepsilon$-regularity results), they proved that the one-dimensional Hausdorff measure of the singular set is $0$. Later on, Lin \cite{lin1998New} also gave a simplified proof. Their $\varepsilon$-regularity criteria are stated as follows:

\vspace{0.4em}

\noindent \textbf{Theorem I.}
Suppose $(u,p)$ is a suitable weak solution to \eqref{eq: NS} in the neighbourhood of $(0,0)$. There exists a small positive constant $\varepsilon$, such that if either
\begin{equation*}
    \int_{-1}^0 \int_{B_1} |u|^3+|p|^{3/2} \, \d x \d t < \varepsilon
\end{equation*}
or
\begin{equation*}
    \varlimsup_{r \to 0} \frac{1}{r} \int_{-r^2}^0 \int_{B_r} |\nabla u|^2 \, \d x \d t < \varepsilon
\end{equation*}
holds, then $(0,0)$ is a regular point.

\vspace{0.4em}

Here and in what follows, a space-time point $(x_0,t_0)$ is said to be regular if the solution is bounded in $Q_r(x_0,t_0)=B_r(x_0) \times (t_0-r^2,t_0)$ for some $r>0$, otherwise it is called a singular point. We'd also like to remark that the first $\varepsilon$-regularity criterion in Theorem I requires only one radius (thus the name ``one scale criterion''), i.e., the assumption is only made in $Q_1$ (or other $Q_{r_0}$ for some fixed $r_0$), whereas the second assumption is required to hold for any small $r$ (``all scales'').

The methods of \cite{caffarelli1982Partial,lin1998New} were widely adopted in local regularity theories of Navier--Stokes equations, see, e.g., \cite{tian1999Gradient,ladyzhenskaya1999Partial,seregin2002Local,seregin2004Boundary,seregin2007Local}, and we recommend readers to refer to the monograph \cite{seregin2014Lecture} for a detailed instruction. Subsequent papers also made attempts to show the local regularity under the smallness assumptions on other scaled quantities, especially the scaled mixed $L_t^s L_x^q$ norms, which generalised the corresponding results of \cite{caffarelli1982Partial,lin1998New}. As far as we know, one of an early attempts was due to He \cite{he2004Partial}, who proved that any suitable weak solution $u$ to \eqref{eq: NS} is bounded near $(0,0)$, provided that either the scaled $L^q$ norm ($1 \leq q \leq 10/3$) of $u$ or the scaled $L^q$ norm ($1 \leq q \leq 2$) of $\nabla u$ is sufficiently small in $Q_r$ for any small $r$. Later, Zajaczkowski and Seregin \cite{zajaczkowski2006Sufficient} considered the case where the scaled $L_t^s L_x^q$ norm of $u$ is small at all scales, with $1 \leq s,q \leq \infty$ satisfying
\begin{equation}
    \label{eq: 2/s+3/q}
    \frac{1}{2} + \max \bigg\{ \frac{1}{2s}, \frac{1}{2}-\frac{1}{q}, \frac{1}{q}-\frac{1}{6} \bigg\} < \frac{2}{s}+\frac{3}{q}-1 \leq 1;
\end{equation}
whereas Gustafson et al. \cite{gustafson2007Interior} dropped the lower bound assumption of $2/s+3/q-1$ in \eqref{eq: 2/s+3/q}. \cite{gustafson2007Interior} also proved the local regularity under the condition that the scaled $L_t^s L_x^q(Q_r)$ norm ($0 \leq 2/s+3/q-2 \leq 1$, $1 \leq s \leq \infty$) of $\nabla u$ or the scaled $L_t^s L_x^q(Q_r)$ norm ($0 \leq 2/s+3/q-2 \leq 1$, $1 \leq s \leq \infty$, $(s,q) \neq (\infty,1)$) of $\curl u$ is sufficiently small for arbitrary small $r$. For other $\varepsilon$-regularity criteria related to \eqref{eq: NS}, either in terms of scaled usual or anisotropic $L_t^s L_x^q$ norms, and either at one scale or all scales, we refer readers to \cite{wang2014Interior,wang2016Anisotropic,he2019New,wang2020ERegularity,li2024Remarks} and references therein.

Similar results have also been obtained for suitable weak solutions to MHD equations, see, e.g., \cite{he2005Partial,vyalov2008Partial,kang2009Interior,wang2013Interior}. To be specific, the triplet $(u,b,p)$ is said to be a suitable weak solution to \eqref{eq: MHD} in the unit parabolic ball $Q_1=Q_1(0,0)$, if

\vspace{0.3em}

(i) $u,b \in L_t^\infty L_x^2 (Q_1)$, $\nabla u, \nabla b \in L^2(Q_1)$, $p \in L^{3/2}(Q_1)$;

(ii) $(u,b,p)$ satisfies \eqref{eq: MHD} in $Q_1$ in the sense of distribution;

(iii) for a.e. $t \in (-1,0)$, the local energy inequality
\begin{equation}
    \label{eq: LEQ}
    \begin{split}
    &\int_{B_1} \phi(x,t) \big( |u(x,t)|^2+|b(x,t)|^2 \big) \d x + 2 \int_{-1}^t \int_{B_1} \phi \big( |\nabla u|^2+|\nabla b|^2 \big) \d x \d \tau \\
    \leq& \int_{-1}^t \int_{B_1} \big( |u|^2+|b|^2 \big) \big( \partial_t \phi+\Delta \phi \big) \d x \d \tau + \int_{-1}^t \int_{B_1} \big( u \cdot \nabla \phi \big) \big( |u|^2+|b|^2+2p \big) \d x \d \tau \\
    &- 2 \int_{-1}^t \int_{B_1} \big( b \cdot \nabla \phi \big) \big( u \cdot b \big) \d x \d \tau
    \end{split}
\end{equation}
holds for any smooth non-negative function $\phi$ vanishing in the vicinity of the parabolic boundary of $Q_1$. To illustrate related results for suitable weak solutions to \eqref{eq: MHD} and our results in this paper, define the following energy quantities which are invariant under the natural scaling $u^\lambda(x,t)=\lambda u(\lambda x,\lambda^2 t)$, $b^\lambda(x,t)=\lambda b(\lambda x,\lambda^2 t)$, $p^\lambda(x,t)=\lambda^2 p(\lambda x,\lambda^2 t)$:
\begin{gather*}
    A(f;r):=\frac{1}{r} \sup_{-r^2<t<0} \int_{B_r} |f(x,t)|^2 \d x, \quad E(f;r):=\frac{1}{r} \iint_{Q_r} |\nabla f|^2 \d x \d t, \\
    C(f;r):=\frac{1}{r^2} \iint_{Q_r} |f|^3 \d x \d t, \quad H(f;r):=\frac{1}{r^3} \iint_{Q_r} |f|^2 \d x \d t
\end{gather*}
for $f=u$ or $b$, and
\begin{equation*}
    D(p;r):=\frac{1}{r^2} \iint_{Q_r} \big| p-[p]_{B_r} \big|^{3/2} \, \d x \d t;
\end{equation*}
also, denote the scaled $L_t^s L_x^q$ norms
\begin{equation*}
    G(f,s,q;r):=r^{1-2/s-3/q} \|f\|_{L_t^s L_x^q(Q_r)}
\end{equation*}
for $f=u$ or $b$, and
\begin{equation*}
    K(f,s,q;r):=r^{2-2/s-3/q} \|f\|_{L_t^s L_x^q(Q_r)}
\end{equation*}
for $f=\nabla u$, $\nabla b$, $\curl u$ or $\curl b$. Here $[\cdot]_{B_r}$ denotes the mean value over the ball $B_r$. For convenience, let $A(u,b;r)=A(u;r)+A(b;r)$ and $E(u,b;r)$, $C(u,b;r)$, $H(u,b;r)$, etc. denote similar notations. When no confusion arises, we also abbreviate $G(f,s,q;r)$ and $K(f,s,q;r)$ as $G(f;r)$ and $K(f;r)$, respectively.

Following similar arguments, the CKN criteria for \eqref{eq: NS} in \cite{caffarelli1982Partial,lin1998New} can be extended to \eqref{eq: MHD}. See Vyalov \cite{vyalov2008Partial}, where it was shown that the smallness of $C(u,b;1)+D(p;1)$ or $\varlimsup_{r \to 0} E(u,b;r)$ or $\varlimsup_{r \to 0} A(u,b;r)$ guarantees the local regularity. Moreover, He and Xin \cite{he2005Partial} obtained the all scales $\varepsilon$-regularity criteria analogous to Navier--Stokes equations in \cite{tian1999Gradient}. Their statement reads as follows:

\vspace{0.4em}

\noindent \textbf{Theorem II.}
(\cite[Proposition 7.1]{he2005Partial}) \ Suppose $(u,b,p)$ is a suitable weak solution to \eqref{eq: MHD} in $Q_1$. There exists a small positive constant $\varepsilon$, such that if one of the following conditions holds:

\vspace{0.1em}

\noindent (i) Either $\varlimsup_{r \to 0} A(u,b;r)<\infty$ or $\varlimsup_{r \to 0} E(u,b;r)<\infty$, and $\varlimsup_{r \to 0} H(u,b;r)<\varepsilon$;

\vspace{0.1em}

\noindent (ii) $\varlimsup_{r \to 0} E(u,b;r)<\varepsilon$;

\vspace{0.1em}

\noindent (iii) $\varlimsup_{r \to 0} r^{1-5/q} (\|u\|_{L^q(Q_r)}+\|b\|_{L^q(Q_r)})<\varepsilon$ for some $5/2 \leq q \leq 10/3$,

\vspace{0.1em}

\noindent then $(0,0)$ is a regular point.

\vspace{0.4em}

By deriving a series of estimates on the scaled energy quantities which reduce their cases to Theorem II, \cite[Theorems 2.1--2.3]{he2005Partial} further showed the local regularity in the cases where certain scaled energy quantities related to $u$ are small, and certain scaled energy quantities related to $b$ are bounded, for arbitrary small $r$. That again implies the dominance of $u$ in the regularity theory.

Regarding $\varepsilon$-regularity criteria in terms of the scaled $L_t^s L_x^q$ norms, Kang and Lee \cite{kang2009Interior} proved that

\vspace{0.4em}

\noindent \textbf{Theorem III.}
Suppose $(u,b,p)$ is a suitable weak solution to \eqref{eq: MHD} in $Q_1$. For any $\mathcal{N}>0$, there exists a small positive constant $\varepsilon=\varepsilon(\mathcal{N})$, such that if one of the following three conditions holds for $b$:
\vspace{0.1em}

\noindent (i) $\varlimsup_{r \to 0} G(b,s_2,q_2;r)<\mathcal{N}$ for some $s_2,q_2$ satisfying $0 \leq 2/s_2+3/q_2-1 \leq 1$, $1 \leq s_2 \leq \infty$;

\vspace{0.1em}

\noindent (ii) $\varlimsup_{r \to 0} K(\nabla b,s_2,q_2;r)<\mathcal{N}$ for some $s_2,q_2$ satisfying $0 \leq 2/s_2+3/q_2-2 \leq 1$, $1 \leq s_2 \leq \infty$;

\vspace{0.1em}

\noindent (iii) $\varlimsup_{r \to 0} K(\curl b,s_2,q_2;r)<\mathcal{N}$ for some $s_2,q_2$ satisfying $0 \leq 2/s_2+3/q_2-2 \leq 1$, $1 \leq s_2 \leq \infty$ and $(s_2,q_2) \neq (\infty,1)$;

\vspace{0.1em}

\noindent and one of the following three conditions holds for $u$:
\vspace{0.1em}

\noindent (i') $\varlimsup_{r \to 0} G(u,s_1,q_1;r)<\varepsilon$ for some $s_1,q_1$ satisfying $0 \leq 2/s_1+3/q_1-1 \leq 1$, $1 \leq s_1 \leq \infty$;

\vspace{0.1em}

\noindent (ii') $\varlimsup_{r \to 0} K(\nabla u,s_1,q_1;r)<\varepsilon$ for some $s_1,q_1$ satisfying $0 \leq 2/s_1+3/q_1-2 \leq 1$, $1 \leq s_1 \leq \infty$;

\vspace{0.1em}

\noindent (iii') $\varlimsup_{r \to 0} K(\curl u,s_1,q_1;r)<\varepsilon$ for some $s_1,q_1$ satisfying $0 \leq 2/s_1+3/q_1-2 \leq 1$, $1 \leq s_1 \leq \infty$ and $(s_1,q_1) \neq (\infty,1)$;

\vspace{0.1em}

\noindent then $(0,0)$ is a regular point.

\vspace{0.4em}

\cite{kang2009Interior} contained much of the results in \cite{he2005Partial} as special cases, and the proof was based on delicate iterations involving several auxiliary scaling invariant quantities. One open question proposed in \cite{kang2009Interior} was whether the assumption on $b$ can be dropped. The following theorem shown by Wang and Zhang \cite{wang2013Interior} gave a positive answer.

\vspace{0.4em}

\noindent \textbf{Theorem IV.}
Suppose $(u,b,p)$ is a suitable weak solution to \eqref{eq: MHD} in $Q_1$. There exists a small positive constant $\varepsilon$, such that if one of the following three conditions holds:

\vspace{0.1em}

\noindent (i) $\varlimsup_{r \to 0} G(u,s,q;r)<\varepsilon$ for some $s,q$ satisfying $0 \leq 2/s+3/q-1 \leq 1$, $1 \leq s \leq \infty$ and $(s,q) \neq (1,\infty)$;

\vspace{0.1em}

\noindent (ii) $\varlimsup_{r \to 0} K(\curl u,s,q;r)<\varepsilon$ for some $s,q$ satisfying $0 \leq 2/s+3/q-2 \leq 1$, $2<s \leq \infty$ and $(s,q) \neq (\infty,1)$;

\vspace{0.1em}

\noindent (iii) $\varlimsup_{r \to 0} K(\curl u,s,q;r)<\varepsilon$ for some $s,q$ satisfying $0 \leq 2/s+3/q-2 \leq 1$, $1<s \leq 2$, and $\sup_{0<r<r_0} G(u,s_0,1,r)<\infty$ for some $r_0 \leq 1$ and $s_0>2$,

\vspace{0.1em}

\noindent then $(0,0)$ is a regular point.

\vspace{0.4em}

The key point in \cite{wang2013Interior} to get rid of the assumption on $b$ is to exploit the term $\iint (u \cdot \nabla \phi) (|b|^2+2p)$ in the local energy inequality \eqref{eq: LEQ} so that $b \cdot \nabla b$ can be controlled in terms of the scaling invariant quantities related to $u$.

The interior $\varepsilon$-regularity criteria in \cite{vyalov2008Partial,kang2009Interior,wang2013Interior} can also be extended to the boundary cases, see \cite{vyalov2010Boundary,vyalov2011Local,vyalov2014Regularity,kang2014Boundary,kim2016Local}, and references therein. It is worth mentioning that Kang and Kim \cite{kang2014Boundary} also gave a different proof of Theorem IV independently.

In this paper, we aim to derive $\varepsilon$-regularity criteria for \eqref{eq: MHD} in terms of the classical scaled energy quantities $A,E,C$ which arise naturally from the local energy inequality, and then to obtain equivalent $\varepsilon$-regularity criteria formulated in terms of the scaled $L_t^s L_x^q$ norms, via suitable interpolation and embedding arguments. As an intermediate step, we establish a characterisation of Type I singularities within the same framework, which is itself of independent interest. Consider the quantity
\begin{equation*}
    g(u):=\min \bigg\{ \varlimsup_{r \to 0} A(u;r), \, \varlimsup_{r \to 0} E(u;r), \, \varlimsup_{r \to 0} C(u;r) \bigg\}.
\end{equation*}
It is known that for suitable weak solutions to \eqref{eq: NS} in $Q_1$, if $g$ is sufficiently small, then $(0,0)$ is a regular point. See, for instance, Seregin \cite{seregin2007Local}. Motivated by this, we are going to show the following theorem.
\begin{theorem}
\label{thm: smallness}
Suppose $(u,b,p)$ is a suitable weak solution to \eqref{eq: MHD} in $Q_1$. If $g(b)<\mathcal{N}<\infty$, and there exists a small positive constant $\varepsilon=\varepsilon(\mathcal{N})$, such that $g(u)<\varepsilon$, then $(0,0)$ is a regular point.
\end{theorem}
Theorem \ref{thm: smallness} is a natural extension of analogous result on Navier--Stokes equations, and allows flexibility in that the assumptions imposed on the scaled energy quantities of $u$ and $b$ are chosen independently.

A natural question is that, instead of the conditions in Theorem \ref{thm: smallness}, whether the boundedness of $g(u)+g(b)$ allows blowup or not. This is still open, even for Navier--Stokes equations. Nevertheless, it can be proved that the boundedness condition can rule out Type II singularities. By the definition in \cite{seregin2010Weak}, a singular point, say $(0,0)$, of a suitable weak solution $(u,p)$ to \eqref{eq: NS}, is said to be of Type I, if
\begin{equation*}
    \varlimsup_{r \to 0} \, [A(u;r)+E(u;r)+C(u;r)+D(p;r)]<\infty,
\end{equation*}
otherwise it is said to be of Type II. It has been shown in \cite{seregin2006Estimates} that for a suitable weak solution $(u,p)$ to \eqref{eq: NS}, the condition $g(u)<\infty$ ensures that the potential singular point $(0,0)$ can only be of Type I. Inspired by that, we'd like to say a singular point $(0,0)$ of a suitable weak solution $(u,b,p)$ to \eqref{eq: MHD} is of Type I, if
\begin{equation}
    \label{eq: TypeI}
    \varlimsup_{r \to 0} \, [A(u,b;r)+E(u,b;r)+C(u,b;r)+D(p;r)]<\infty
\end{equation}
We will show the following result analogous to \cite{seregin2006Estimates}, which characterises Type I singularities of \eqref{eq: MHD} under much weaker assumptions.

\begin{theorem}
\label{thm: boundedness}
Suppose $(u,b,p)$ is a suitable weak solution of \eqref{eq: MHD} in $Q_1$. If $g(u)<\infty$ and $g(b)<\infty$, then \eqref{eq: TypeI} holds.
\end{theorem}
As mentioned before, Theorem \ref{thm: smallness} recovers Theorem III proved by Kang and Lee \cite{kang2009Interior}, as a consequence of interpolation and embedding, which we will discuss in Section \ref{sec: mixed}. In Section \ref{sec: mixed}, similar arguments also lead to the following characterisation of Type I singularities in terms of the scaled $L_t^s L_x^q$ norms.
\begin{theorem}
\label{thm: boundedness_mixed}
Suppose $(u,b,p)$ is a suitable weak solution to \eqref{eq: MHD} in $Q_1$. If one of the following three conditions holds for $b$:

\vspace{0.1em}

\noindent (i) $\varlimsup_{r \to 0} G(b,s_2,q_2;r)<\infty$ for some $s_2,q_2$ satisfying $0 \leq 2/s_2+3/q_2-1 \leq 1$, $1 \leq s_2 \leq \infty$;

\vspace{0.1em}

\noindent (ii) $\varlimsup_{r \to 0} K(\nabla b,s_2,q_2;r)<\infty$ for some $s_2,q_2$ satisfying $0 \leq 2/s_2+3/q_2-2 \leq 1$, $1 \leq s_2 \leq \infty$;

\vspace{0.1em}

\noindent (iii) $\varlimsup_{r \to 0} K(\curl b,s_2,q_2;r)<\infty$ for some $s_2,q_2$ satisfying $0 \leq 2/s_2+3/q_2-2 \leq 1$, $1 \leq s_2 \leq \infty$ and $(s_2,q_2) \neq (\infty,1)$;

\vspace{0.1em}

\noindent and one of the following three conditions holds for $u$:

\vspace{0.1em}

\noindent (i') $\varlimsup_{r \to 0} G(u,s_1,q_1;r)<\infty$ for some $s_1,q_1$ satisfying $0 \leq 2/s_1+3/q_1-1<1$, $1 \leq s_1 \leq \infty$;

\vspace{0.1em}

\noindent (ii') $\varlimsup_{r \to 0} K(\nabla u,s_1,q_1;r)<\infty$ for some $s_1,q_1$ satisfying $0 \leq 2/s_1+3/q_1-2<1$, $1<s_1 \leq \infty$;

\vspace{0.1em}

\noindent (iii') $\varlimsup_{r \to 0} K(\curl u,s_1,q_1;r)<\infty$ for some $s_1,q_1$ satisfying $0 \leq 2/s_1+3/q_1-2<1$, $1<s_1 \leq \infty$;

\vspace{0.1em}

\noindent then \eqref{eq: TypeI} holds.
\end{theorem}

We do not claim to improve existing $\varepsilon$-regularity criteria in the literature. Nevertheless, our formulation in terms of the scaled energy quantities $A,E$ and $C$ provides a direct and transparent approach compared to the arguments in \cite{kang2009Interior} and \cite{wang2013Interior}, and it is naturally compatible with the Type I framework. In particular, merely replacing the smallness parameter $\varepsilon$ in the proofs of \cite{kang2009Interior} and \cite{wang2013Interior} by an arbitrary finite bound does not lead to a Type I characterisation, whereas our approach yields such a result in a unified way. We believe that this Type I description is of independent interest, and it also offers a natural starting point for the study of Type II blowup, in the spirit of Seregin's recent works on the Navier--Stokes equations \cite{seregin2024Remarks,seregin2024Note}.

The remainder of the paper is organised as follows. In Section \ref{sec: dimensionless}, we will derive some dimensionless estimates that will be useful in subsequent discussions. Based on these dimensionless estimates and standard iteration arguments, we will show Theorem \ref{thm: boundedness} in Section \ref{sec: boundedness}, and by exploiting the decay of the heat kernel that appears in the integral representation of the magnetic field $b$, we will prove Theorem \ref{thm: smallness} in Section \ref{sec: smallness}. Finally, in Section \ref{sec: mixed}, we will show how the $\varepsilon$-regularity criteria and the characterisation of Type I singularities formulated in terms of $A,E$ and $C$ extend naturally to the scaled $L_t^s L_x^q$ norms framework and give a proof of Theorem III and Theorem \ref{thm: boundedness_mixed}.

\section{Some dimensionless estimates}
\label{sec: dimensionless}
In this section, we present several estimates of the scaled energy quantities which will be useful in later derivations. Let $(u,b,p)$ be a suitable weak solution of \eqref{eq: MHD} in $Q_1$. First we have the following variation of the local energy inequality:
\begin{equation}
    \label{eq: LEQ1}
    A(u,b;r)+E(u,b;r) \lesssim H(u,b;2r)+C(u;2r)+\frac{1}{r^2} \iint_{Q_{2r}} |b|^2 |u| + \frac{1}{r^2} \iint_{Q_{2r}} |p-[p]_{B_{2r}}||u|
\end{equation}
for any $0<r<2r \leq 1$, which can be obtained by selecting $\phi$ in \eqref{eq: LEQ} such that $\phi=0$ in the vicinity of the parabolic boundary of $Q_{2r}$, $\phi=1$ in $Q_r$, and $|\nabla \phi| \lesssim r^{-1}$, $|\partial_t \phi|+|\nabla^2 \phi| \lesssim r^{-2}$ in $Q_{2r}$.

The following lemma gives estimates for $C$ in terms of other scaled energy quantities.
\begin{lemma}
For $f=u$ or $b$ and any $0<r \leq \rho$, we have
\begin{gather}
    C(f;r) \lesssim A^{1/2}(f;r) [H^{1/4}(f;r) E^{3/4}(f;r)+H(f;r)], \label{eq: C<A(HE+H)} \\
    C(f;r) \lesssim \bigg(\frac{r}{\rho}\bigg)^3 A^{3/2}(f;\rho) + \bigg(\frac{\rho}{r}\bigg)^{3/2} A^{3/4}(f;\rho) E^{3/4}(f;\rho). \label{eq: C<A+AE}
\end{gather}
\end{lemma}
\begin{proof}
The proofs can be found in \cite[Lemma 4.1]{he2005Partial} and \cite[Lemma 2.1]{lin1998New}, respectively, but for the completeness, we still present them here. By Gagliardo--Nirenberg inequality, we have
\begin{equation}
    \label{eq: GN}
    \int_{B_r} |f|^3 \d x \lesssim \bigg( \int_{B_r} |f|^2 \d x \bigg)^{3/4} \bigg( \int_{B_r} |\nabla f|^2 \d x \bigg)^{3/4} + r^{-3/2} \bigg( \int_{B_r} |f|^2 \d x \bigg)^{3/2}.
\end{equation}
Integrating in time, we obtain by H\"older's inequality that
\begin{equation*}
    \begin{split}
    \iint_{Q_r} |f|^3 \lesssim& \sup_t \bigg( \int_{B_r} |f|^2 \d x \bigg)^{1/2} \bigg( \iint_{Q_r} |f|^2 \bigg)^{1/4} \bigg( \iint_{Q_r} |\nabla f|^2 \bigg)^{3/4} \\
    &+ r^{-3/2} \sup_t \bigg( \int_{B_r} |f|^2 \d x \bigg)^{1/2} \iint_{Q_r} |f|^2 \\
    \lesssim& r^2 A^{1/2}(f;r) [H^{1/4}(f;r) E^{3/4}(f;r)+H(f;r)],
    \end{split}
\end{equation*}
which implies \eqref{eq: C<A(HE+H)}. On the other hand, by Poincar\'e's inequality, we have
\begin{equation*}
    \begin{split}
    \int_{B_r} |f|^2 \d x \lesssim& \rho \bigg(\int_{B_\rho} \big| |f|^2-[|f|^2]_{B_\rho} \big|^{3/2} \, \d x\bigg)^{2/3} + \int_{B_r} [|f|^2]_{B_\rho} \d x \\
    \lesssim& \rho \int_{B_\rho} |f| |\nabla f| \d x + \bigg(\frac{r}{\rho}\bigg)^3 \int_{B_\rho} |f|^2 \d x \\
    \lesssim& \rho \bigg(\int_{B_\rho} |f|^2 \d x\bigg)^{1/2} \bigg(\int_{B_\rho} |\nabla f|^2 \d x\bigg)^{1/2} + \bigg(\frac{r}{\rho}\bigg)^3 \int_{B_\rho} |f|^2 \d x.
    \end{split}
\end{equation*}
Applying this to the last term of \eqref{eq: GN}, and integrating in time, we obtain by H\"older's inequality that
\begin{equation*}
    \begin{split}
    \iint_{Q_r} |f|^3 \lesssim& \sup_t \bigg(\int_{B_r} |f|^2 \d x\bigg)^{3/4} \cdot r^{1/2} \bigg(\iint_{Q_r} |\nabla f|^2 \bigg)^{3/4} \\
    &+ \bigg(\frac{\rho}{r}\bigg)^{3/2} \sup_t \bigg(\int_{B_\rho} |f|^2 \d x\bigg)^{3/4} \cdot r^{1/2} \bigg(\iint_{Q_\rho} |\nabla f|^2 \bigg)^{3/4} \\
    &+ r^{-3/2} \cdot \bigg(\frac{r}{\rho}\bigg)^{9/2} \cdot r^2 \sup_t \bigg(\int_{B_\rho} |f|^2 \d x\bigg)^{3/2} \\
    \lesssim& \bigg[ \rho^{3/2} r^{1/2} + \bigg(\frac{\rho}{r}\bigg)^{3/2} r^2 \bigg] A^{3/4}(f;\rho) E^{3/4}(f;\rho) + \bigg(\frac{r}{\rho}\bigg)^3 r^2 A^{3/2}(f;\rho),
    \end{split}
\end{equation*}
which implies \eqref{eq: C<A+AE}.
\end{proof}

Next we need to derive some decay estimates of $D(p;r)$. Let $0<r<2r \leq \rho \leq 1$. By taking the divergence of \eqref{eq: MHD}$_1$, we get
\begin{equation}
    -\Delta p(\cdot,t) = \divergence \divergence (u \otimes u) - \divergence \divergence (b \otimes b) \quad \text{in $B_\rho$}
\end{equation}
in the sense of distribution for a.e. $t \in (-\rho^2,0)$. Decompose $p$ as $p=p_1+p_2+p_3$, where for a.e. $t \in (-\rho^2,0)$,
\begin{equation}
    \label{eq: Delta(p1p2)}
    \begin{gathered}
    \int_{B_\rho} p_1(x,t) \Delta \phi(x) \d x = -\int_{B_\rho} (u \otimes u) : \nabla^2 \phi \d x, \\
    \int_{B_\rho} p_2(x,t) \Delta \phi(x) \d x = \int_{B_\rho} (b \otimes b) : \nabla^2 \phi \d x
    \end{gathered}
\end{equation}
for any $\phi \in W^{2,3}(B_\rho)$ with $\phi|_{\partial B_\rho}=0$, and
\begin{equation*}
    \Delta p_3(\cdot,t)=0 \quad \text{in $B_\rho$}
\end{equation*}
in the sense of distribution. By Calder\'on--Zygmund estimate, we have
\begin{equation}
    \label{eq: p1p2(1)}
    \int_{B_\rho} |p_1|^{3/2} \d x \lesssim \int_{B_\rho} |u \otimes u|^{3/2} \d x \lesssim \int_{B_\rho} |u|^3 \d x, \quad \int_{B_\rho} |p_2|^{3/2} \d x \lesssim \int_{B_\rho} |b|^3 \d x,
\end{equation}
and as has been shown in \cite{seregin2006Estimates}, by the harmonicity of $p_3$ in $B_\rho$ we have
\begin{equation*}
    \begin{split}
    \sup_{x \in B_r} \big| p_3(x,t)-[p_3]_{B_r}(t) \big| \lesssim& r \sup_{x \in B_{\rho/2}} |\nabla p_3(x,t)| \lesssim r \cdot \frac{1}{\rho^4} \int_{B_\rho} \big| p_3(x,t)-[p_3]_{B_\rho}(t) \big| \\
    \lesssim& \frac{r}{\rho} \cdot \frac{1}{\rho^2} \bigg( \int_{B_\rho} \big| p_3(x,t) - [p_3]_{B_\rho}(t) \big|^{3/2} \, \d x \bigg)^{2/3}.
    \end{split}
\end{equation*}
Therefore, by \eqref{eq: p1p2(1)} we have
\begin{equation}
    \label{eq: p3}
    \begin{split}
    D(p_3;r) \lesssim& r \int_{-r^2}^0 \sup_{x \in B_r} \big| p_3(x,t)-[p_3]_{B_r}(t) \big|^{3/2} \d t \\
    \lesssim& \bigg(\frac{r}{\rho}\bigg)^{5/2} \cdot \frac{1}{\rho^2} \iint_{Q_\rho} \big| p_3(x,t)-[p_3]_{B_\rho}(t) \big|^{3/2} \\
    =& \bigg(\frac{r}{\rho}\bigg)^{5/2} D(p_3;\rho) \lesssim \bigg(\frac{r}{\rho}\bigg)^{5/2} [D(p;\rho)+D(p_1;\rho)+D(p_2;\rho)] \\
    \lesssim& \bigg(\frac{r}{\rho}\bigg)^{5/2} [D(p;\rho)+C(u,b;\rho)].
    \end{split}
\end{equation}
Combining \eqref{eq: p1p2(1)} with \eqref{eq: p3}, we obtain
\begin{equation}
    \label{eq: p(1)}
    D(p;r) \lesssim \bigg(\frac{r}{\rho}\bigg)^{5/2} D(p;\rho) + \bigg(\frac{\rho}{r}\bigg)^2 C(u,b;\rho).
\end{equation}

In this paper, we need some other decay estimates of $D(p;r)$. By replacing $u \otimes u$ on the right hand side of \eqref{eq: Delta(p1p2)}$_1$ with $(\tilde{u} \otimes \tilde{u}-[\tilde{u} \otimes \tilde{u}]_{B_\rho})$, where $\tilde{u}:=u-[u]_{B_\rho}$, we get by Calder\'on--Zygmund estimate and Poincar\'e inequality that
\begin{equation}
    \label{eq: p1p2(2)}
    \begin{split}
    \int_{B_\rho} |p_1|^{3/2} \d x \lesssim& \int_{B_\rho} \big| \tilde{u} \otimes \tilde{u}-[\tilde{u} \otimes \tilde{u}]_{B_\rho} \big|^{3/2} \d x \lesssim \bigg( \int_{B_\rho} \big| \nabla (\tilde{u} \otimes \tilde{u}) \big| \d x \bigg)^{3/2} \\
    \lesssim& \bigg( \int_{B_\rho} |\nabla \tilde{u}| |\tilde{u}| \d x \bigg)^{3/2} = \bigg( \int_{B_\rho} |\nabla u| \big| u-[u]_{B_\rho} \big| \d x \bigg)^{3/2} \\
    \lesssim& \bigg( \int_{B_\rho} |\nabla u|^2 \d x \bigg)^{3/4} \bigg( \int_{B_\rho} \big| u-[u]_{B_\rho} \big|^2 \d x \bigg)^{3/4},
    \end{split}
\end{equation}
where
\begin{equation*}
    \begin{split}
    \bigg( \int_{B_\rho} \big| u-[u]_{B_\rho} \big|^2 \d x \bigg)^{3/4} =& \bigg( \int_{B_\rho} \big| u-[u]_{B_\rho} \big|^2 \d x \bigg)^{1/2} \bigg( \int_{B_\rho} \big| u-[u]_{B_\rho} \big|^2 \d x \bigg)^{1/4} \\
    \lesssim& \bigg( \int_{B_\rho} |u|^2 \d x \bigg)^{1/2} \cdot \rho^{1/2} \bigg( \int_{B_\rho} |\nabla u|^2 \d x \bigg)^{1/4}.
    \end{split}
\end{equation*}
Hence,
\begin{equation*}
    D(p_1;\rho) \lesssim A^{1/2}(u;\rho) E(u;\rho).
\end{equation*}
On the other hand, we can also directly integrate \eqref{eq: p1p2(2)} in $t$ and apply H\"older inequality to get
\begin{equation*}
    D(p_1;\rho) \lesssim \frac{1}{\rho^2} \bigg( \sup_{-\rho^2<t<0} \int_{B_\rho} |u|^2 \d x \bigg)^{3/4} \int_{-\rho^2}^0 \bigg( \int_{B_\rho} |\nabla u|^2 \d x \bigg)^{3/4} \d t \lesssim A^{3/4}(u;\rho) E^{3/4}(u;\rho).
\end{equation*}
The same estimates also hold for $D(p_2;\rho)$. In this way, we've shown the following lemma.
\begin{lemma}
For any $0<2r \leq \rho \leq 1$, we have
\begin{gather}
    D(p;r) \lesssim \bigg(\frac{r}{\rho}\bigg)^{5/2} D(p;\rho) + \bigg(\frac{\rho}{r}\bigg)^2 \big[ A^{1/2}(u;\rho) E(u;\rho)+C(b;\rho) \big], \label{eq: p(2)} \\
    D(p;r) \lesssim \bigg(\frac{r}{\rho}\bigg)^{5/2} D(p;\rho) + \bigg(\frac{\rho}{r}\bigg)^2 \big[ C(u;\rho) + A^{1/2}(b;\rho) E(b;\rho) \big], \label{eq: p(3)} \\
    D(p;r) \lesssim \bigg(\frac{r}{\rho}\bigg)^{5/2} D(p;\rho) + \bigg(\frac{\rho}{r}\bigg)^2 \big[ A^{3/4}(u;\rho)E^{3/4}(u;\rho)+A^{3/4}(b;\rho)E^{3/4}(b;\rho) \big]. \label{eq: p(4)}
\end{gather}
\end{lemma}

\section{The boundedness of scaled quantities}
\label{sec: boundedness}
Let $(u,b,p)$ be a suitable weak solution of \eqref{eq: MHD} in $Q_1$, and let $\mathcal{E}(r):=A(u,b;r)+E(u,b;r)+C(u,b;r)+D(p;r)$. Theorem \ref{thm: boundedness} follows from the following three propositions.
\begin{proposition}
\label{thm: A(u)<M}
There exist absolute constants $\tilde{c},\alpha>0$ and $0<r_0<1$, such that for arbitrary positive constants $\mathcal{M},\mathcal{N}$, if
\begin{equation*}
    \sup_{0<r<1} A(u;r) \leq \mathcal{M}, \ \text{and} \ \min \bigg\{ \sup_{0<r<1} A(b;r), \, \sup_{0<r<1} E(b;r), \, \sup_{0<r<1} C(b;r) \bigg\} \leq \mathcal{N},
\end{equation*}
then
\begin{equation}
    \label{eq: mathcalE(r)}
    \mathcal{E}(r) \leq \tilde{c} r^\alpha \mathcal{E}(1)+\mathcal{G}(\tilde{c},\mathcal{M},\mathcal{N}), \quad \forall 0<r<r_0,
\end{equation}
where $\mathcal{G}$ is continuous with respect to $\mathcal{M}$ and $\mathcal{N}$, and $\mathcal{G}(\tilde{c},\mathcal{M},\mathcal{N}) \to 0$ as $\mathcal{M},\mathcal{N} \to 0$.
\end{proposition}
\begin{proof}
Let $0<4r \leq \rho \leq 1$. By the local energy inequality \eqref{eq: LEQ1}, we have
\begin{equation}
    \label{eq: LEQ2}
    \mathcal{E}(r) \lesssim H(u,b;2r)+C(u,b;2r)+D(p;2r).
\end{equation}
By \eqref{eq: C<A(HE+H)}, we have
\begin{equation*}
    \begin{split}
    C(u;2r) \leq& \bigg(\frac{\rho}{r}\bigg)^2 C(u;\rho) \leq c_0 \bigg(\frac{\rho}{r}\bigg)^2 A^{1/2}(u;\rho) \big[ H^{1/4}(u;\rho) E^{3/4}(u;\rho) + H(u;\rho) \big] \\
    \leq& c_0 \bigg(\frac{\rho}{r}\bigg)^2 \mathcal{M}^{1/2} \big[ C^{1/6}(u;\rho) E^{3/4}(u;\rho) + C^{2/3}(u;\rho) \big] \\
    \leq& \delta \big[ C(u;\rho)+E(u;\rho) \big] + d(\delta) \bigg[ \bigg(\frac{\rho}{r}\bigg)^{24} \mathcal{M}^6 + \bigg(\frac{\rho}{r}\bigg)^6 \mathcal{M}^{3/2} \bigg].
    \end{split}
\end{equation*}
Here and in what follows, $c_0$ denotes a positive constant independent of the quantities we concern about and may vary from line to line; $\delta$ denotes a small positive number to be determined later and $d(\delta)$ is a positive number depending on $\delta$. Similarly, we have by H\"older inequality and \eqref{eq: C<A+AE} that
\begin{equation*}
    \begin{split}
    H(u;2r) \leq c_0 C^{2/3}(u,b;2r) \leq c_0 \bigg(\frac{r}{\rho}\bigg)^2 \mathcal{M}+\bigg(\frac{\rho}{r}\bigg) \mathcal{M}^{1/2} E^{1/2}(u;\rho).
    \end{split}
\end{equation*}
In what follows, we consider three cases. In the case where $A(b;r) \leq \mathcal{N}$, $C(b;2r)$ and $H(b;2r)$ can be estimated in the same way as $u$, and by \eqref{eq: p(4)} we have
\begin{equation*}
    \begin{split}
    D(p;2r) \leq& c_0 \bigg(\frac{r}{\rho}\bigg)^{5/2} D(p;\rho) + \bigg(\frac{\rho}{r}\bigg)^2 \big[ \mathcal{M}^{3/4} E^{3/4}(u;\rho) + \mathcal{N}^{3/4} E^{3/4}(b;\rho) \big] \\
    \leq& c_0 \bigg(\frac{r}{\rho}\bigg)^{5/2} D(p;\rho) + \delta E(u,b;\rho) + d(\delta) \bigg(\frac{\rho}{r}\bigg)^8 \big( \mathcal{M}^3+\mathcal{N}^3 \big).
    \end{split}
\end{equation*}
In the case where $E(b;r) \leq \mathcal{N}$, we derive by \eqref{eq: C<A(HE+H)} and Poincar\'e inequality that
\begin{equation*}
    \begin{split}
    C \big( b-[b]_{B_\rho};2r \big) \leq& c_0 \bigg(\frac{\rho}{r}\bigg)^{7/2} A^{1/2}(b;\rho) \big[ H^{1/4} \big( b-[b]_{B_\rho};\rho \big) E^{3/4}(b;\rho) + H \big( b-[b]_{B_\rho};\rho \big) \big] \\
    \leq& c_0 \bigg(\frac{\rho}{r}\bigg)^{7/2} A^{1/2}(b;\rho) E(b;\rho) \leq \delta A(b;\rho) + d(\delta) \bigg(\frac{\rho}{r}\bigg)^7 \mathcal{N}^2,
    \end{split}
\end{equation*}
which, combined with H\"older inequality, yields that
\begin{equation}
    \label{eq: C(b)}
    \begin{split}
    C(b;2r) \leq& c_0 C \big( b-[b]_{B_\rho};2r \big) + c_0 C \big( [b]_{B_\rho};2r \big) \\
    \leq& \delta A(b;\rho) + d(\delta) \bigg(\frac{\rho}{r}\bigg)^7 \mathcal{N}^2 + c_0 \bigg(\frac{r}{\rho}\bigg) C(b;\rho).
    \end{split}
\end{equation}
In addition, through \eqref{eq: C<A+AE}, we have
\begin{equation}
    \label{eq: H(b)}
    \begin{split}
    H(b;2r) \leq& c_0 C^{2/3}(b;2r) \leq c_0 \bigg[ \bigg(\frac{r}{\rho}\bigg)^2 A(b;\rho) + \bigg(\frac{\rho}{r}\bigg) A^{1/2}(b;\rho) \mathcal{N}^{1/2} \bigg] \\
    \leq& c_0 \bigg(\frac{r}{\rho}\bigg)^2 A(b;\rho) + \delta A(b;\rho) + d(\delta) \bigg(\frac{\rho}{r}\bigg)^2 \mathcal{N},
    \end{split}
\end{equation}
and by \eqref{eq: p(4)}, we have
\begin{equation*}
    \begin{split}
    D(p;2r) \leq& c_0 \bigg(\frac{r}{\rho}\bigg)^{5/2} D(p;\rho) + c_0 \bigg(\frac{\rho}{r}\bigg)^2 \big[ \mathcal{M}^{3/4} E^{3/4}(u;\rho) + \mathcal{N}^{3/4} A^{3/4}(b;\rho) \big] \\
    \leq& c_0 \bigg(\frac{r}{\rho}\bigg)^{5/2} D(p;\rho) + \delta \big[ E(u;\rho)+A(b;\rho) \big] + d(\delta) \bigg(\frac{\rho}{r}\bigg)^8 \big( \mathcal{M}^3 + \mathcal{N}^3 \big).
    \end{split}
\end{equation*}
In the case where $C(b;r) \leq \mathcal{N}$, we have immediately that $H(b;2r) \leq c_0 \mathcal{N}^{2/3}$, and $D(p;2r)$ can be estimated using \eqref{eq: p(1)}.

Therefore, in any case, we obtain
\begin{equation}
    \label{eq: mathcalE(r)_recursive}
    \mathcal{E}(r) \leq \bigg[\delta+c_0 \bigg(\frac{r}{\rho}\bigg)\bigg] \mathcal{E}(\rho) + d(\delta) \bigg(\frac{\rho}{r}\bigg)^\gamma \mathcal{G}_0 (\mathcal{M},\mathcal{N})
\end{equation}
for some positive $\gamma$ and any $0<4r \leq \rho \leq 1$, where $\mathcal{G}_0$ is continuous with respect to $\mathcal{M}$ and $\mathcal{N}$, and $\mathcal{G}_0 \to 0$ as $\mathcal{M},\mathcal{N} \to 0$. If we denote $r/\rho=\theta$, and fix $\theta$ and $\delta$ so that $c_0 \theta^{1/2} \leq 1/2$, $\theta \leq 1/4$, $\delta \leq \theta^{1/2}/2$, then
\begin{equation}
    \label{eq: mathcalE(r)_recursive2}
    \mathcal{E}(r) \leq \theta^{1/2} \mathcal{E}(\theta^{-1} r)+d(\delta) \theta^{-\gamma} \mathcal{G}_0 (\mathcal{M},\mathcal{N}), \quad \forall 0<r \leq \theta.
\end{equation}
Iterating \eqref{eq: mathcalE(r)_recursive2} for $k$ times, where the positive integer $k$ satisfies $\theta^{-k} r \leq 1<\theta^{-(k+1)} r$, we obtain
\begin{equation*}
    \begin{split}
    \mathcal{E}(r) \leq& \theta^{1/2} \mathcal{E}(\theta^{-1} r) + d(\delta) \theta^{-\gamma} \mathcal{G}_0(\mathcal{M},\mathcal{N}) \\
    \leq& \theta \mathcal{E}(\theta^{-2} r) + d(\delta) \theta^{-\gamma} \big(1+\theta^{1/2}\big) \mathcal{G}_0(\mathcal{M},\mathcal{N}) \\
    \leq& \cdots \leq \theta^{k/2} \mathcal{E}(\theta^{-k} r) + \frac{d(\delta) \theta^{-\gamma}}{1-\theta^{1/2}} \mathcal{G}_0(\mathcal{M},\mathcal{N}) \\
    \leq& \theta^{k/2} \cdot \frac{1}{(\theta^{-k} r)^2} \mathcal{E}(1) + \frac{d(\delta) \theta^{-\gamma}}{1-\theta^{1/2}} \mathcal{G}_0(\mathcal{M},\mathcal{N}), \quad \forall \theta^{k+1} < r \leq \theta^k.
    \end{split}
\end{equation*}
Noting that $r>\theta^{k+1}$ implies that $\theta^{k/2}/(\theta^{-k} r)^2<\theta^{-5/2} r^{1/2}$, we obtain
\begin{equation*}
    \mathcal{E}(r) \leq \theta^{-5/2} r^{1/2} \mathcal{E}(1) + \mathcal{G}(\theta,\delta,\mathcal{M},\mathcal{N}), \quad \forall 0<r \leq \theta,
\end{equation*}
where $\mathcal{G} (\theta,\delta,\mathcal{M},\mathcal{N}) = d(\delta)\theta^{-\gamma} \mathcal{G}_0 (\mathcal{M},\mathcal{N})/(1-\theta^{1/2})$.
\end{proof}
\begin{proposition}
\label{thm: E(u)<M}
There exist absolute constants $\tilde{c},\alpha>0$ and $0<r_0<1$, such that for arbitrary positive constants $\mathcal{M},\mathcal{N}$, if
\begin{equation*}
    \sup_{0<r<1} E(u;r) \leq \mathcal{M}, \ \text{and} \ \min \bigg\{ \sup_{0<r<1} A(b;r), \, \sup_{0<r<1} E(b;r), \, \sup_{0<r<1} C(b;r) \bigg\} \leq \mathcal{N},
\end{equation*}
then \eqref{eq: mathcalE(r)} holds for any $0<r<r_0$.
\end{proposition}
\begin{proof}
Let $0<4r \leq \rho \leq 1$. $C(u;2r)$ and $H(u;2r)$ can be estimated in the same way as \eqref{eq: C(b)} and \eqref{eq: H(b)}, respectively. The rest part of the proof is divided into three cases. The case $E(u;r) \leq \mathcal{M}$, $A(b;r) \leq \mathcal{N}$ is basically the same as $A(u;r) \leq \mathcal{M}$, $E(b;r) \leq \mathcal{N}$ in Proposition \ref{thm: A(u)<M}. In the case where $E(b;r) \leq \mathcal{N}$, the estimates of $C(b;2r)$ and $H(b;2r)$ are given by \eqref{eq: C(b)} and \eqref{eq: H(b)}, respectively; whereas by \eqref{eq: p(4)}, we have
\begin{equation*}
    \begin{split}
    D(p;2r) \leq& c_0 \bigg(\frac{r}{\rho}\bigg)^{5/2} D(p;\rho) + c_0 \bigg(\frac{\rho}{r}\bigg)^2 \big[ A^{3/4}(u;\rho) \mathcal{M}^{3/4} + A^{3/4}(b;\rho) \mathcal{N}^{3/4} \big] \\
    \leq& c_0 \bigg(\frac{r}{\rho}\bigg)^{5/2} D(p;\rho) + \delta A(u,b;\rho) + d(\delta) \bigg(\frac{\rho}{r}\bigg)^8 \big( \mathcal{M}^3 + \mathcal{N}^3 \big).
    \end{split}
\end{equation*}
In this way we obtain the estimates of the right-hand-side of \eqref{eq: LEQ2}. In the case where $C(b;r) \leq \mathcal{N}$, we have $H(b;2r) \leq c_0 \mathcal{N}^{2/3}$, and
\begin{equation*}
    \frac{1}{r^2} \iint_{Q_{2r}} |b|^2 |u| \leq \bigg(\frac{\rho}{r}\bigg)^{2/3} \mathcal{N}^{2/3} C^{1/3}(u;\rho) \leq \delta C(u;\rho)+d(\delta) \bigg(\frac{\rho}{r}\bigg) \mathcal{N}.
\end{equation*}
Also, recalling \eqref{eq: p(2)}, we have
\begin{equation*}
    \begin{split}
    &\frac{1}{r^2} \iint_{Q_{2r}} \big| p-[p]_{B_{2r}} \big| |u| \leq D^{2/3}(p;2r) C^{1/3}(u;2r) \\
    \leq& c_0 \bigg(\frac{r}{\rho}\bigg) D^{2/3}(p;\rho) C^{1/3}(u;\rho) + c_0 \bigg(\frac{\rho}{r}\bigg)^2 C^{1/3}(u;\rho) \big[ A^{1/3}(u;\rho) \mathcal{M}^{2/3} + \mathcal{N}^{2/3} \big] \\
    \leq& c_0 \bigg(\frac{r}{\rho}\bigg) \mathcal{E}(\rho) + \delta [A(u;\rho)+C(u;\rho)] + d(\delta) \bigg[ \bigg(\frac{\rho}{r}\bigg)^6 \mathcal{M}^2 + \bigg(\frac{\rho}{r}\bigg)^3 \mathcal{N} \bigg],
    \end{split}
\end{equation*}
and
\begin{equation*}
    \begin{split}
    D(p;2r) \leq& c_0 \bigg(\frac{r}{\rho}\bigg)^{5/2} D(p;\rho) + c_0 \bigg(\frac{\rho}{r}\bigg)^2 \big[ A^{1/2}(u;\rho) \mathcal{M}+\mathcal{N} \big] \\
    \leq& c_0 \bigg(\frac{r}{\rho}\bigg)^{5/2} D(p;\rho) + \delta A(u;\rho) + d(\delta) \bigg(\frac{\rho}{r}\bigg)^4 \mathcal{M}^2 + c_0 \bigg(\frac{\rho}{r}\bigg)^2 \mathcal{N}.
    \end{split}
\end{equation*}
It then suffices to apply \eqref{eq: LEQ1}. In any case, we obtain an inequality in the form of \eqref{eq: mathcalE(r)_recursive} in Proposition \ref{thm: A(u)<M}. The remaining proof is similar.
\end{proof}
\begin{proposition}
\label{thm: C(u)<M}
There exist absolute constants $\tilde{c},\alpha>0$ and $0<r_0<1$, such that for arbitrary positive constants $\mathcal{M},\mathcal{N}$, if
\begin{equation*}
    \sup_{0<r<1} C(u;r) \leq \mathcal{M}, \ \text{and} \ \min \bigg\{ \sup_{0<r<1} A(b;r), \, \sup_{0<r<1} E(b;r), \, \sup_{0<r<1} C(b;r) \bigg\} \leq \mathcal{N},
\end{equation*}
then \eqref{eq: mathcalE(r)} holds for any $0<r<r_0$.
\end{proposition}
\begin{proof}
The case $C(u;r) \leq \mathcal{M}$, $A(b;r) \leq \mathcal{N}$ is basically the same as $A(u;r) \leq \mathcal{M}$, $C(b;r) \leq \mathcal{N}$ in Proposition \ref{thm: A(u)<M}. The case $C(u;r) \leq \mathcal{M}$, $E(b;r) \leq \mathcal{N}$ is almost the same as $E(u;r) \leq \mathcal{M}$, $C(b;r) \leq \mathcal{N}$ in Proposition \ref{thm: E(u)<M}, except that
\begin{equation*}
    \frac{1}{r^2} \iint_{Q_{2r}} \big| p-[p]_{B_{2r}} \big| |u| \leq c_0 \bigg(\frac{\rho}{r}\bigg)^{4/3} \mathcal{M}^{1/3} D^{2/3}(p;\rho) \leq \delta D(p;\rho) + d(\delta) \bigg(\frac{\rho}{r}\bigg)^4 \mathcal{M}.
\end{equation*}
Finally, in the case where $C(b;r) \leq \mathcal{N}$, we deduce from \eqref{eq: LEQ1} and \eqref{eq: p(1)} that
\begin{equation*}
    \begin{split}
    \mathcal{E}(r) \leq& c_0 \bigg[ H(u,b;2r)+\bigg(\frac{\rho}{r}\bigg)^2 C(u,b;\rho) + \bigg(\frac{r}{\rho}\bigg)^{5/2} D(p;\rho) \bigg] \\
    \leq& c_0 \bigg[ \mathcal{M}^{2/3}+\mathcal{N}^{2/3}+\bigg(\frac{\rho}{r}\bigg)^2 (\mathcal{M}+\mathcal{N}) + \bigg(\frac{r}{\rho}\bigg)^{5/2} D(p;\rho) \bigg].
    \end{split}
\end{equation*}
In any case, we arrive at an inequality similar to \eqref{eq: mathcalE(r)_recursive}, and the desired result follows likewise.
\end{proof}

\section{The smallness of scaled quantities}
\label{sec: smallness}
In this section we apply the boundedness estimates obtained in Section \ref{sec: boundedness} to establish the $\varepsilon$-regularity criteria stated in Theorem \ref{thm: smallness}.
\begin{proposition}
\label{thm: C(u)_small}
For arbitrary $\varepsilon_0>0$ and $\mathcal{N}>0$, there exists $\varepsilon=\varepsilon(\varepsilon_0,\mathcal{N})>0$, such that if
\begin{equation}
    \label{eq: C(u)_small}
    \varlimsup_{r \to 0} C(u;r) \leq \varepsilon, \ \text{and} \ \min \bigg\{ \sup_{0<r<1} A(b;r), \, \sup_{0<r<1} E(b;r), \, \sup_{0<r<1} C(b;r) \bigg\} \leq \mathcal{N},
\end{equation}
then
\begin{equation*}
    \varlimsup_{r \to 0} H(b;r) \leq \varepsilon_0.
\end{equation*}
\end{proposition}
\begin{proof}
For any $0<\rho \leq 1$, let $\chi=\chi(x,t)$ and $\varphi(x,t)$ be arbitrary smooth scalar and vector functions, respectively, and suppose $\varphi$ is compactly supported in $Q_\rho$, and $\chi$ vanishes near the parabolic boundary of $Q_\rho$. Testing \eqref{eq: MHD}$_2$ with $\chi \varphi$, we obtain
\begin{equation*}
    \iint_{Q_\rho} b \cdot \big[ \partial_t (\chi \varphi) + \Delta (\chi \varphi) \big] = -\iint_{Q_\rho} (u \otimes b-b \otimes u) : \nabla (\chi \varphi),
\end{equation*}
that is,
\begin{equation*}
    \begin{split}
    \iint_{Q_\rho} b \chi \cdot (\partial_t \varphi+\Delta \varphi) =& -\iint_{Q_\rho} b \cdot \varphi (\partial_t \chi-\Delta \chi) + 2 \big[ b \cdot \varphi \Delta \chi + (\nabla \chi \otimes b) : \nabla \varphi \big] \\
    &-\iint_{Q_\rho} \big( u \otimes b-b \otimes u \big) : \nabla (\chi \varphi) \\
    =& -\iint_{Q_\rho} b \cdot \varphi (\partial_t \chi-\Delta \chi) - 2 (\nabla \chi \otimes \varphi) : \nabla b \\
    &+\iint_{Q_\rho} \chi \big[ (u \otimes \varphi) : \nabla b - (b \otimes \varphi) : \nabla u \big],
    \end{split}
\end{equation*}
which means
\begin{equation}
    \label{eq: d(bchi)/dt-Delta(bchi)}
    \partial_t (b \chi) - \Delta (b \chi) = b (\partial_t \chi-\Delta \chi) - 2 (\nabla \chi \cdot \nabla) b - \chi (u \cdot \nabla) b + \chi (b \cdot \nabla) u \quad \text{in} \ Q_\rho
\end{equation}
in the sense of distribution. Suppose $\chi \equiv 1$ in $Q_{\rho/2}$, and $|\nabla \chi| \lesssim |\rho|^{-1}$, $|\partial_t \chi|+|\nabla^2 \chi| \lesssim |\rho|^{-2}$. For any $(x,t) \in Q_{\rho/2}$, we have by \eqref{eq: d(bchi)/dt-Delta(bchi)} that
\begin{equation}
    \label{eq: bchi}
    \begin{split}
    b \chi (x,t) =& \int_{-\rho^2}^t \int_{B_\rho} \Gamma(x-y,t-s) \big[ b (\partial_t \chi-\Delta \chi) - 2 (\nabla \chi \cdot \nabla) b \big](y,s) \, \d y \d s \\
    &- \int_{-\rho^2}^t \int_{B_\rho} \Gamma(x-y,t-s) \chi (u \cdot \nabla) b(y,s) \, \d y \d s \\
    &+ \int_{-\rho^2}^t \int_{B_\rho} \Gamma(x-y,t-s) \chi (b \cdot \nabla) u(y,s) \, \d y \d s,
    \end{split}
\end{equation}
where $\Gamma$ is the heat kernel and satisfies the well-known pointwise estimate (see, e.g., \cite{solonnikov1964Estimates})
\begin{equation}
    \label{eq: heat_kernel}
    \big| \nabla^k \partial_t^l \Gamma(x,t) \big| \lesssim \big( |x|^2+t \big)^{-(3+k+2l)/2}.
\end{equation}
By integration by parts, we obtain
\begin{equation*}
    \begin{split}
    |b \chi(x,t)| \leq& \frac{c_0}{\rho^2} \int_{-\rho^2}^t \int_{B_\rho \backslash B_{\rho/2}} \Gamma(x-y,t-s) |b| \, \d y \d s \\
    &+ \frac{c_0}{\rho} \int_{-\rho^2}^t \int_{B_\rho \backslash B_{\rho/2}} |\nabla \Gamma(x-y,t-s)| |b| \, \d y \d s \\
    &+ c_0 \int_{-\rho^2}^t \int_{B_\rho} |\nabla \Gamma(x-y,t-s)| |u| |b| \, \d y \d s \\
    &+ \frac{c_0}{\rho} \int_{-\rho^2}^t \int_{B_\rho \backslash B_{\rho/2}} \Gamma(x-y,t-s) |u| |b| \, \d y \d s \\
    =:& (I_1+I_2+I_3+I_4)(x,t).
    \end{split}
\end{equation*}
Let $0<4r \leq \rho \leq 1$ and $(x,t) \in Q_r$. Recalling \eqref{eq: heat_kernel}, we have
\begin{gather*}
    (I_1+I_2)(x,t) \leq \frac{c_0}{\rho^5} \int_{-\rho^2}^t \int_{B_\rho} |b(y,s)| \, \d y \d s \leq \frac{c_0}{\rho} H^{1/2}(b;\rho), \\
    I_4(x,t) \leq \frac{c_0}{\rho^4} \int_{-\rho^2}^t \int_{B_\rho} |u(y,s)| |b(y,s)| \, \d y \d s \leq \frac{c_0}{\rho} C^{1/3}(u;\rho) H^{1/2}(b;\rho).
\end{gather*}
Also, we can deduce by setting $Y:=y/(t-s)^{1/2}$ that
\begin{equation*}
    \int_{\mathbb{R}^3} \frac{1}{(|y|^2+t-s)^{12/5}} \, \d y \leq \frac{1}{(t-s)^{9/10}} \int_{\mathbb{R}^3} \frac{1}{(|Y|^2+1)^3} \, \d Y \leq \frac{c_0}{(t-s)^{9/10}}.
\end{equation*}
Therefore, by applying Minkowski's integral inequality (see, e.g., \cite[Appendices, A.1]{stein1970Singular}), Young's convolution inequality and \eqref{eq: heat_kernel}, we derive
\begin{equation*}
    \begin{split}
    \|I_3(\cdot,t)\|_{L^2(B_r)} =& c_0 \bigg( \int_{B_r} \bigg( \int_{-\rho^2}^t \int_{B_\rho} |\nabla \Gamma(x-y,t-s)| |u| |b| \, \d y \, \d s \bigg)^2 \, \d x \bigg)^{1/2} \\
    \leq& c_0 \int_{-\rho^2}^t \bigg( \int_{B_r} \bigg( \int_{B_\rho} |\nabla \Gamma(x-y,t-s)| |u| |b| \, \d y \bigg)^2 \, \d x \bigg)^{1/2} \, \d s \\
    \leq&  c_0 \int_{-\rho^2}^t \|\nabla \Gamma(\cdot,t-s)\|_{L^{6/5}(\mathbb{R}^3)} \|u(\cdot,s)\|_{L^3(B_\rho)} \|b(\cdot,s)\|_{L^3(B_\rho)} \, \d s \\
    \leq& c_0 \int_{-\rho^2}^t \bigg( \int_{\mathbb{R}^3} \frac{1}{(|y|^2+t-s)^{12/5}} \, \d y \bigg)^{5/6} \|u(\cdot,s)\|_{L^3(B_\rho)} \|b(\cdot,s)\|_{L^3(B_\rho)} \, \d s \\
    \leq& c_0 \int_{-\rho^2}^t \frac{1}{(t-s)^{3/4}} \|u(\cdot,s)\|_{L^3(B_\rho)} \|b(\cdot,s)\|_{L^3(B_\rho)} \, \d s,
    \end{split}
\end{equation*}
and thus, by Young's convolution inequality, we obtain
\begin{equation*}
    \|I_3\|_{L^2(Q_r)} \leq c_0 \rho^{1/6} \|u\|_{L^3(Q_\rho)} \|b\|_{L^3(Q_\rho)} = c_0 \rho^{3/2} C^{1/3}(u;\rho) C^{1/3}(b;\rho).
\end{equation*}
As a consequence,
\begin{equation*}
    \iint_{Q_r} |b|^2 \, \d x \d t \leq c_0 \frac{r^5}{\rho^2} \big[ H(b;\rho) + C^{2/3}(u;\rho) H(b;\rho) \big] + c_0 \rho^3 C^{2/3}(u;\rho) C^{2/3}(b;\rho),
\end{equation*}
which yields that
\begin{equation}
    \label{eq: H(b)_recursive1}
    H(b;r) \leq c_0 \bigg(\frac{r}{\rho}\bigg)^2 H(b;\rho) + c_0 \bigg(\frac{\rho}{r}\bigg)^3 C^{2/3}(u;\rho) \big[ C^{2/3}(b;\rho)+H(b;\rho) \big].
\end{equation}
Recalling \eqref{eq: C(u)_small}$_1$, there exists $r_0 \leq 1$, such that $\sup_{0<r<r_0} C(u;r) \leq \varepsilon$. We may as well assume $r_0=1$ without loss of generality. Then by \eqref{eq: H(b)_recursive1} and the results in Section \ref{sec: boundedness} (again, we may as well assume the result \eqref{eq: mathcalE(r)} in Section \ref{sec: boundedness} holds for all $0<r<1$), we obtain
\begin{equation*}
    H(b;r) \leq c_0 \bigg(\frac{r}{\rho}\bigg)^2 H(b;\rho) + \bigg(\frac{\rho}{r}\bigg)^3 \varepsilon^{2/3} \mathcal{G}_1, \quad \forall 0<4r \leq \rho \leq 1,
\end{equation*}
where $\mathcal{G}_1=\mathcal{G}_1(\mathcal{E}(1),\varepsilon,\mathcal{N})$ is continuous with respect to $\varepsilon$ and $\mathcal{N}$. Denote $r/\rho=\theta$ and fix $\theta$ so that $c_0 \theta^{3/2} \leq 1$, $\theta \leq 1/4$, then an iteration process similar to Proposition \ref{thm: A(u)<M} leads to
\begin{equation*}
    H(b;r) \leq \theta^{-5/2} r^{1/2} H(b;1) + \frac{\theta^{-3} \varepsilon^{2/3} \mathcal{G}_1}{1-\theta^{1/2}}
\end{equation*}
for small $r$, which gives the desired result.
\end{proof}
\begin{proposition}
\label{thm: A(u)_small}
For arbitrary $\varepsilon_0>0$ and $\mathcal{N}>0$, there exists $\varepsilon=\varepsilon(\varepsilon_0,\mathcal{N})>0$, such that if
\begin{equation}
    \label{eq: A(u)_small}
    \varlimsup_{r \to 0} A(u;r) \leq \varepsilon, \ \text{and} \ \min \bigg\{ \sup_{0<r<1} A(b;r), \, \sup_{0<r<1} E(b;r), \, \sup_{0<r<1} C(b;r) \bigg\} \leq \mathcal{N},
\end{equation}
then
\begin{equation*}
    \varlimsup_{r \to 0} H(b;r) \leq \varepsilon_0.
\end{equation*}
\end{proposition}
\begin{proof}
Let $0<4r \leq \rho \leq 1$ and $(x,t) \in Q_r$. By \eqref{eq: heat_kernel}, we have
\begin{gather*}
    (I_1+I_2)(x,t) \leq \frac{c_0}{\rho^5} \int_{-\rho^2}^t \int_{B_\rho} |b(y,s)| \, \d y \d s \leq \frac{c_0}{\rho} H^{1/2}(b;\rho), \\
    I_4(x,t) \leq \frac{c_0}{\rho^4} \int_{-\rho^2}^t \int_{B_\rho} |u(y,s)| |b(y,s)| \, \d y \d s \leq \frac{c_0}{\rho} A^{1/2}(u;\rho) H^{1/2}(b;\rho).
\end{gather*}
According to Minkowski's integral inequality, Young's convolution inequality, Gagliardo--Nirenberg inequality and \eqref{eq: heat_kernel}, we have
\begin{equation*}
    \begin{split}
    \|I_3(\cdot,t)\|_{L^2(B_r)} \leq& c_0 \int_{-\rho^2}^t \bigg( \int_{B_r} \bigg( \int_{B_\rho} |\nabla \Gamma(x-y,t-s)| |u| |b| \, \d y \bigg)^2 \, \d x \bigg)^{1/2} \, \d s \\
    \leq&  c_0 \int_{-\rho^2}^t \|\nabla \Gamma(\cdot,t-s)\|_{L^{6/5}(\mathbb{R}^3)} \|u(\cdot,s)\|_{L^2(B_\rho)} \|b(\cdot,s)\|_{L^6(B_\rho)} \, \d s \\
    \leq& c_0 \int_{-\rho^2}^t \bigg( \int_{\mathbb{R}^3} \frac{1}{(|y|^2+t-s)^{12/5} \, \d y} \bigg)^{5/6} \|u(\cdot,s)\|_{L^2(B_\rho)} \|b(\cdot,s)\|_{L^6(B_\rho)} \, \d s \\
    \leq& c_0 \rho^{1/2} A^{1/2}(u;\rho) \int_{-\rho^2}^t \frac{1}{(t-s)^{3/4}} \big( \|\nabla b(\cdot,s)\|_{L^2(B_\rho)} + \rho^{-1} \|b(\cdot,s)\|_{L^2(B_\rho)} \big) \, \d s.
    \end{split}
\end{equation*}
Hence, by Young's convolution inequality, we obtain
\begin{equation*}
    \begin{split}
    \|I_3\|_{L^2(Q_r)} \leq& c_0 \rho^{1/2} A^{1/2}(u;\rho) \rho^{1/10} \big( \|\nabla b\|_{L^2(Q_\rho)} + \rho^{-1} \|b\|_{L^2(Q_\rho)} \big) \rho^{2/5} \\
    =& c_0 \rho^{3/2} A^{1/2}(u;\rho) \big[ E^{1/2}(b;\rho)+H^{1/2}(b;\rho) \big].
    \end{split}
\end{equation*}
As a result,
\begin{equation*}
    \begin{split}
    \iint_{Q_r} |b|^2 \, \d x \d t \leq& c_0 \frac{r^5}{\rho^2} \big[ H(b;\rho) + A(u;\rho) H(b;\rho) \big] + c_0 \rho^3 A(u;\rho) \big[ E(b;\rho)+H(b;\rho) \big],
    \end{split}
\end{equation*}
and thus
\begin{equation*}
    H(b;r) \leq c_0 \bigg(\frac{r}{\rho}\bigg)^2 H(b;\rho) + c_0 \bigg(\frac{\rho}{r}\bigg)^3 A(u;\rho) \big[ E(b;\rho)+H(b;\rho) \big].
\end{equation*}
Recalling the smallness assumption on $A(u;\rho)$ and the boundedness of $E(b;\rho)+H(b;\rho)$ which follows from Section \ref{sec: boundedness}, we arrive at the desired result by implementing the iteration process.
\end{proof}
\begin{proposition}
\label{thm: E(u)_small}
For arbitrary $\varepsilon_0>0$ and $\mathcal{N}>0$, there exists $\varepsilon=\varepsilon(\varepsilon_0,\mathcal{N})>0$, such that if
\begin{equation}
    \label{eq: E(u)_small}
    \varlimsup_{r \to 0} E(u;r) \leq \varepsilon, \ \text{and} \ \min \bigg\{ \sup_{0<r<1} A(b;r), \, \sup_{0<r<1} E(b;r), \, \sup_{0<r<1} C(b;r) \bigg\} \leq \mathcal{N},
\end{equation}
then
\begin{equation*}
    \varlimsup_{r \to 0} H(b;r) \leq \varepsilon_0.
\end{equation*}
\end{proposition}
\begin{proof}
Let $0<4r \leq \rho \leq 1$. Through an argument almost the same as Proposition \ref{thm: A(u)_small}, we can derive
\begin{equation}
    \label{eq: H(b)_recursive2}
    H(b;r) \leq c_0 \bigg(\frac{r}{\rho}\bigg)^2 H(b;\rho) + c_0 \bigg(\frac{\rho}{r}\bigg)^3 A(b;\rho) \big[ E(u;\rho)+H(u;\rho) \big].
\end{equation}
Denote $\theta=r/\rho$. By \eqref{eq: C<A+AE}, we have
\begin{equation*}
    H(u;\rho) \leq c_0 C^{2/3}(u;\rho) \leq c_0 \theta^4 A(u;\theta^{-2} \rho) + c_0 \theta^{-2} A^{1/2}(u;\theta^{-2} \rho) E^{1/2}(u;\theta^{-2} \rho)
\end{equation*}
for $\rho \leq \theta^2$, which, combined with \eqref{eq: H(b)_recursive2}, yields
\begin{equation*}
    \begin{split}
    H(b;r) \leq& c_0 \theta^2 H(b;\theta^{-1} r) + c_0 \theta^{-3} A(b;\theta^{-1} r) E(u;\theta^{-1} r) \\
    &+ c_0 \theta A(b;\theta^{-1} r) A(u;\theta^{-3} r) + c_0 \theta^{-5} A(b;\theta^{-1} r) A^{1/2}(u;\theta^{-3} r) E^{1/2}(u;\theta^{-3} r)
    \end{split}
\end{equation*}
for $r \leq \theta^3$. Again, recalling \eqref{eq: E(u)_small}$_1$, we may as well assume $\sup_{0<r<1} E(u;r) \leq \varepsilon$, then the above estimate and the results in Section \ref{sec: boundedness} lead to
\begin{equation}
    \label{eq: H(b)_recursive3}
    H(b;r) \leq c_0 \theta^2 H(b;\theta^{-1} r) + \big( \theta^{-3} \varepsilon + \theta + \theta^{-5} \varepsilon^{1/2} \big) \mathcal{G}_1,
\end{equation}
where $\mathcal{G}_1=\mathcal{G}_1(\mathcal{E}(1),\varepsilon,\mathcal{N})$ is continuous with respect to $\varepsilon$ and $\mathcal{N}$. Suppose $\theta$ is small enough so that $c_0 \theta \leq 1$, $\theta \leq 1/4$. Iterating \eqref{eq: H(b)_recursive3} for $k$ times, where the positive integer $k$ satisfies $\theta^{-k} r \leq \theta^2<\theta^{-k-1} r$, we obtain
\begin{equation*}
    \begin{split}
    H(b;r) \leq& \theta^k H(b;\theta^{-k} r) + \big( \theta^{-3} \varepsilon + \theta + \theta^{-5} \varepsilon^{1/2} \big) \cdot \frac{\mathcal{G}_1}{1-\theta} \\
    \leq& \theta^k \cdot \bigg(\frac{\theta^2}{\theta^{-k} r}\bigg)^3 H(b;\theta^2) + \big( \theta^{-3} \varepsilon + \theta + \theta^{-5} \varepsilon^{1/2} \big) \cdot \frac{\mathcal{G}_1}{1-\theta} \\
    \leq& \theta^{-6} r H(b;\theta^2) + \big( \theta^{-3} \varepsilon + \theta + \theta^{-5} \varepsilon^{1/2} \big) \cdot \frac{\mathcal{G}_1}{1-\theta}
    \end{split}
\end{equation*}
for small $r$. We then complete the proof by taking first $\theta$ and then $\varepsilon,r$ small enough.
\end{proof}

According to \eqref{eq: H(b)} and the results in Section \ref{sec: boundedness}, it is easy to see that under the assumptions of Propositions \ref{thm: C(u)_small}, \ref{thm: A(u)_small} and \ref{thm: E(u)_small}, we can deduce
\begin{equation*}
    \varlimsup_{r \to 0} H(u;r) \leq \varepsilon^{2/3}, \quad \varlimsup_{r \to 0} H(u;r) \leq \varepsilon, \quad \text{and} \ \varlimsup_{r \to 0} H(u;r) \leq \theta \mathcal{G}_2 + d(\theta^{-1}) \theta^{-2} \varepsilon,
\end{equation*}
respectively, where $\mathcal{G}_2(\mathcal{E}(1),\varepsilon,\mathcal{N})$ is continuous with respect to $\varepsilon$ and $\mathcal{N}$, and $\theta \mathcal{G}_2 + d(\theta^{-1}) \theta^{-2} \varepsilon<\varepsilon_0$ if we take first $\theta$ and then $\varepsilon$ small. This, combined with the results of Propositions \ref{thm: C(u)_small}--\ref{thm: E(u)_small} and Theorem II (i) in the Introduction part, completes the proof of Theorem \ref{thm: smallness}. \hfill$\Box$

\section{Scaled mixed norms}
\label{sec: mixed}
We now apply interpolation and embedding to the scale-invariant bounds derived earlier to obtain $\varepsilon$-regularity and Type I characterisation in terms of the mixed Lebesgue norms.

\subsection{Mixed Lebesgue norms of \texorpdfstring{$u$ and $b$}{u and b}} \hfill \\
\label{subsec: (i)}

Let $f=u$ or $b$, and $1 \leq s,q \leq \infty$. Assume $0 \leq 2/s+3/q-1 \leq 1$. It follows from interpolation and Gagliardo--Nirenberg inequalities that
\begin{equation*}
    \|f\|_{L^3(B_r)} \leq \|f\|_{L^q(B_r)}^\alpha \|f\|_{L^2(B_r)}^\beta \big( \|\nabla f\|_{L^2(B_r)}+r^{-1} \|f\|_{L^2(B_r)} \big)^\gamma,
\end{equation*}
where $0 \leq \alpha,\beta,\gamma \leq 1$ satisfy
\begin{equation}
    \label{eq: parameter1}
    \frac{\alpha}{q}+\frac{\beta}{2}+\frac{\gamma}{6}=\frac{1}{3}, \quad \alpha+\beta+\gamma=1.
\end{equation}
Therefore, by integrating in $t$ and applying H\"older inequality, we derive
\begin{equation}
    \label{eq: C(f)}
    \begin{split}
    C(f;r) \leq& c_0 r^{-2(3\alpha/s+3\gamma/2)} \bigg( \sup_{-r^2<t<0} \int_{B_r} |f|^2 \bigg)^{3\beta/2} \|f\|_{L_t^s L_x^q}^{3\alpha} \bigg( \iint_{Q_r} |\nabla f|^2+r^{-2} |f|^2 \bigg)^{3\gamma/2} \\
    \leq& c_0 A^{3\beta/2}(f;r) G^{3\alpha}(f;r) \big[ E(f;r)+H(f;r) \big]^{3\gamma/2} \\
    \leq& c_0 [A(f;r)+E(f;r)]^\mu G^{3\alpha}(f;r), \quad \mu:=3(1-\alpha)/2,
    \end{split}
\end{equation}
provided that
\begin{equation}
    \label{eq: parameter2}
    \frac{3\alpha}{s}+\frac{3\gamma}{2} \leq 1, \ \ \text{or equivalently,} \ \ 3\alpha \bigg(\frac{2}{s}+\frac{3}{q}-\frac{3}{2}\bigg) \leq \frac{1}{2}, \ \ \text{according to \eqref{eq: parameter1}.}
\end{equation}
In order that \eqref{eq: parameter2} is satisfied, take $\alpha=1/3$, i.e., $\mu=1$, when $2/s+3/q-1=1$; whereas when $2/s+3/q-1<1$, we may set $\alpha>1/3$, i.e., $\mu<1$.

For any $0<4r<\rho<1$, it follows from \eqref{eq: C(f)} and \eqref{eq: LEQ1} that
\begin{equation*}
    \begin{split}
    C(f;r) \leq& c_0 G^{3-2\mu}(f;r) \bigg(\frac{\rho}{r}\bigg)^2 \big[ C^{2/3}(u,b;\rho) + C(u;\rho) \\
    &+ C^{1/3}(u;\rho) C^{2/3}(b;\rho) + C^{1/3}(u;\rho) D^{2/3}(p;\rho) \big]^\mu.
    \end{split}
\end{equation*}
By Young's inequality, when $\mu=1$, we have
\begin{equation*}
    C(f;r) \leq c_0 G(f;r) \bigg(\frac{\rho}{r}\bigg)^2 \big[ \delta_1 C(u,b;\rho) + \delta_1^{-2} + \delta_2^{-2} C(u;\rho) + \delta_2 C(b;\rho) + \delta_2 D(p;\rho) \big];
\end{equation*}
and when $\mu<1$, we have
\begin{equation*}
    \begin{split}
    C(f;r) \leq& c_0 G^{3-2\mu}(f;r) \bigg(\frac{\rho}{r}\bigg)^2 \big[ C^{2/3}(u,b;\rho) + C(u,b;\rho) + D(p;\rho) \big]^\mu \\
    \leq& c_0 \delta_3 [C(u,b;\rho)+D(p;\rho)] + \delta_3^{-\frac{\mu}{1-\mu}} \bigg(\frac{\rho}{r}\bigg)^{\frac{2}{1-\mu}} \big[ G^{\frac{3-2\mu}{1-\mu}}(f;r) + G^3(f;r) \big],
    \end{split}
\end{equation*}
where $\delta_1,\delta_2,\delta_3$ are small parameters to be determined. In the case where $\mu=1$ and only the boundedness condition rather than smallness condition is imposed on $G(f;r)$, we need to multiply $C(f;r)$ by a small parameter so that it can be controlled.

For example, suppose $G(u,s_1,q_1;r) \leq \varepsilon$ and $G(b,s_2,q_2;r) \leq \mathcal{N}$ for any $0<r<1$, where $2/s_1+3/q_1-1=2/s_2+3/q_2-1=1$. Set $\mathcal{E}_0(r):=C(u;r)+\sigma_1 C(b;r)+\sigma_2 D(p;r)$, where $\sigma_1=\varepsilon/\mathcal{N}$, and $\sigma_2$ is a small parameter to be determined later, then it holds by \eqref{eq: p(1)} and the above calculations that
\begin{equation}
    \label{eq: mathcalE(r)_0}
    \begin{split}
    \mathcal{E}_0(r) \leq& c_0 \theta^{-2} (\varepsilon \delta_1 + \varepsilon \delta_2^{-2} + \sigma_2) C(u;\rho) + c_0 \theta^{-2} (\varepsilon \delta_1 + \varepsilon \delta_2 + \sigma_2) C(b;\rho) \\
    &+ c_0 (\varepsilon \theta^{-2} \delta_2 + \sigma_2 \theta^{5/2}) D(p;\rho) + c_0 \varepsilon \theta^{-2} \delta_1^{-2},
    \end{split}
\end{equation}
where $\theta=r/\rho$. If we choose
\begin{equation*}
    \sigma_2=\frac{\theta^3}{4c_0} \sigma_1 = \frac{\theta^3 \varepsilon}{4c_0 \mathcal{N}}, \quad \delta_1=\frac{\theta^3}{4c_0 \varepsilon} \sigma_1 = \frac{\theta^3}{4c_0 \mathcal{N}}, \quad \delta_2=\frac{\theta^3}{4c_0 \varepsilon} \sigma_2 = \frac{\theta^6}{(4c_0)^2 \mathcal{N}},
\end{equation*}
and set
\begin{equation*}
    \varepsilon=\frac{\theta^3 (1-\theta) \varepsilon_0}{4 c_0} \delta_2^2 = \frac{\theta^{15} (1-\theta) \varepsilon_0}{(4 c_0)^5 \mathcal{N}^2}, \quad \theta \leq \frac{1}{4c_0},
\end{equation*}
where $\varepsilon_0$ is the small constant $\varepsilon$ in Theorem \ref{thm: smallness}, then \eqref{eq: mathcalE(r)_0} reduces to
\begin{equation*}
    \mathcal{E}_0(r) \leq \theta \mathcal{E}_0(\rho) + c_0 \theta^{-2} \varepsilon \delta_1^{-2},
\end{equation*}
and a standard iteration process similar to the previous text leads to
\begin{equation*}
    \mathcal{E}_0(r) \leq \theta^{-3} r \mathcal{E}_0(1) + \frac{c_0 \theta^{-2} \varepsilon}{1-\theta} \delta_1^{-2}
\end{equation*}
for small $r$. Recalling our setting for the parameters, it is easy to check that the residual term $=\theta^7 \varepsilon_0/(64 c_0^2) \leq \varepsilon_0$, so $\varlimsup_{r \to 0} \mathcal{E}_0(r) \leq \varepsilon_0$, which implies that
\begin{equation*}
    \varlimsup_{r \to 0} C(u;r) \leq \varepsilon_0, \quad \varlimsup_{r \to 0} C(b;r) \leq \frac{\varepsilon_0}{\sigma_1} = \frac{\varepsilon_0}{\varepsilon} \mathcal{N} = \frac{(4c_0)^5 \mathcal{N}^3}{\theta^{15} (1-\theta)} < \infty.
\end{equation*}
Therefore, it follows from Theorem \ref{thm: smallness} that $(0,0)$ is a regular point.

If we instead assume that $G(u,s_1,q_1;r) \leq \mathcal{M}$ and $G(b,s_2,q_2;r) \leq \mathcal{N}$ for any $0<r<1$, where $0 \leq 2/s_1+3/q_1-1<1$, $2/s_2+3/q_2-1=1$, then
\begin{equation*}
    \begin{split}
    \mathcal{E}_0(r) \leq& c_0 \big( \theta^{-2} \sigma_1 \mathcal{N} \delta_1 + \theta^{-2} \sigma_1 \mathcal{N} \delta_2^{-2} + \theta^{-2} \sigma_2 + \delta_3 \big) C(u;\rho) \\
    &+ c_0 \big( \theta^{-2} \sigma_1 \mathcal{N} \delta_1 + \theta^{-2} \sigma_1 \mathcal{N} \delta_2 + \theta^{-2} \sigma_2 + \delta_3 \big) C(b;\rho) \\
    &+ c_0 \big( \theta^{-2} \sigma_1 \mathcal{N} \delta_2 + \theta^{5/2} \sigma_2 + \delta_3 \big) D(p;\rho) \\
    &+ c_0 \theta^{-2} \sigma_1 \mathcal{N} \delta_1^{-2} + \theta^{-\frac{2}{1-\mu}} \delta_3^{-\frac{\mu}{1-\mu}} \big( \mathcal{M}^{\frac{3-2\mu}{1-\mu}} + \mathcal{M}^3 \big).
    \end{split}
\end{equation*}
Comparing this with the coefficients in \eqref{eq: mathcalE(r)_0}, it is easy to find that if we set
\begin{gather*}
    \sigma_1 = \frac{\theta^{15} (1-\theta) \varepsilon_0}{(4c_0)^5 \mathcal{N}^3}, \quad \sigma_2 = \frac{\theta^3}{4c_0} \sigma_1 = \frac{\theta^{18} (1-\theta) \varepsilon_0}{(4c_0)^6 \mathcal{N}^3}, \\
    \delta_1 = \frac{\theta^3}{4c_0 \mathcal{N}}, \quad \delta_2 = \frac{\theta^6}{(4c_0)^2 \mathcal{N}}, \quad \delta_3=\frac{\theta}{4c_0} \sigma_2 = \frac{\theta^{19} (1-\theta) \varepsilon_0}{(4c_0)^7 \mathcal{N}^3}, \quad \theta \leq \frac{1}{4c_0},
\end{gather*}
then
\begin{equation*}
    \mathcal{E}_0(r) \leq \theta \mathcal{E}_0(\rho) + c_0 \theta^{-2} \sigma_1 \mathcal{N} \delta_1^{-2} + \theta^{-\frac{2}{1-\mu}} \delta_3^{-\frac{\mu}{1-\mu}} \big( \mathcal{M}^{\frac{3-2\mu}{1-\mu}} + \mathcal{M}^3 \big),
\end{equation*}
and a standard iteration process yields that
\begin{equation*}
    \mathcal{E}_0(r) \leq \theta^{-3} r \mathcal{E}_0(1) + \frac{1}{1-\theta} \big[ c_0 \theta^{-2} \sigma_1 \mathcal{N} \delta_1^{-2} + \theta^{-\frac{2}{1-\mu}} \delta_3^{-\frac{\mu}{1-\mu}} \big( \mathcal{M}^{\frac{3-2\mu}{1-\mu}} + \mathcal{M}^3 \big) \big]
\end{equation*}
for small $r$. If we denote the residual term as $\mathcal{R}$, then $\varlimsup_{r \to 0} \mathcal{E}_0(r) \leq \mathcal{R}$, which implies that
\begin{gather*}
    \varlimsup_{r \to 0} C(u;r) \leq \mathcal{R} = \frac{\theta^7 \varepsilon_0}{64c_0^2} + c(c_0,\theta) \varepsilon_0^{-\frac{\mu}{1-\mu}} \mathcal{N}^{\frac{3\mu}{1-\mu}} \big( \mathcal{M}^{\frac{3-2\mu}{1-\mu}} + \mathcal{M}^3 \big), \\
    \varlimsup_{r \to 0} C(b;r) \leq \mathcal{R}/\sigma_1 = \frac{16 c_0^3 \mathcal{N}^3}{\theta^8 (1-
    \theta)} + c(c_0,\theta) \varepsilon_0^{-\frac{1}{1-\mu}} \mathcal{N}^{\frac{3}{1-\mu}} \big( \mathcal{M}^{\frac{3-2\mu}{1-\mu}} + \mathcal{M}^3 \big)
\end{gather*}
under our setting for the parameters. Therefore, the Type I characterisation follows from Theorem \ref{thm: boundedness}; and it holds by Theorem \ref{thm: smallness} that $(0,0)$ is a regular point, if we set $\mathcal{M}=\mathcal{M}(\varepsilon_0,\mathcal{N})$ small enough such that $\mathcal{R} \leq \varepsilon_0$.

The conclusions of Theorem III and Theorem \ref{thm: boundedness_mixed} in terms of $G(u,s_1,q_1;r)$ and $G(b,s_2,q_2;r)$ also hold for the case where both $2/s_1+3/q_1-1$ and $2/s_2+3/q_2-1$ are smaller than $1$, as a consequence of H\"older inequality.

\subsection{Mixed Lebesgue norms of \texorpdfstring{$\nabla u$ and $\nabla b$}{nabla u and nabla b}} \hfill \\
\label{subsec: (ii)}

Similar to Subsection \ref{subsec: (i)}, for any $1 \leq s,q \leq \infty$ with $1 \leq 2/s+3/q-1 \leq 2$, it holds that
\begin{equation*}
    \begin{split}
    \|f-[f]_{B_r}\|_{L^3(B_r)} \leq& c_0 \|f-[f]_{B_r}\|_{L^{q'}(B_r)}^\alpha \|f\|_{L^2(B_r)}^\beta \|f-[f]_{B_r}\|_{L^6(B_r)}^\gamma \\
    \leq& c_0 \|\nabla f\|_{L^q(B_r)}^{\alpha \vartheta} \|f\|_{L^2(B_r)}^{\alpha(1-\vartheta)+\beta} \|\nabla f\|_{L^2(B_r)}^\gamma,
    \end{split}
\end{equation*}
where $0 \leq \alpha,\beta,\gamma,\vartheta \leq 1$ and $1 \leq q' \leq \infty$ satisfy
\begin{equation}
    \label{eq: parameter3}
    \frac{\alpha}{q'}+\frac{\beta}{2}+\frac{\gamma}{6}=\frac{1}{3}, \quad \alpha+\beta+\gamma=1, \quad \frac{1}{q'}=\vartheta \bigg(\frac{1}{q}-\frac{1}{3}\bigg)+\frac{1-\vartheta}{2}.
\end{equation}
Integrating in $t$, we obtain
\begin{equation*}
    \begin{split}
    &\frac{1}{r^2} \iint_{Q_r} \big|f-[f]_{B_r}\big|^3 \\
    \leq& c_0 r^{-2} \bigg( \sup_{-r^2<t<0} \int_{B_r} |f|^2 \bigg)^{3[\alpha(1-\vartheta)+\beta]/2} \|\nabla f\|_{L_t^s L_x^q}^{3\alpha \vartheta} \bigg( \iint_{Q_r} |\nabla f|^2 \bigg)^{3\gamma/2} \cdot r^{2(1-3\alpha \vartheta/s-3\gamma/2)} \\
    \leq& c_0 [A(f;r)+E(f;r)]^\lambda K^{3\alpha \vartheta}(\nabla f;r), \quad \lambda:=\frac{3(1-\alpha \vartheta)}{2},
    \end{split}
\end{equation*}
provided that
\begin{equation}
    \label{eq: parameter4}
    \frac{3\alpha \vartheta}{s}+\frac{3\gamma}{2} \leq 1, \ \ \text{or equivalently,} \ \ 3 \alpha \vartheta \bigg(\frac{2}{s}+\frac{3}{q}-\frac{5}{2}\bigg) \leq \frac{1}{2}.
\end{equation}
In order that \eqref{eq: parameter4} is satisfied, take $\alpha \vartheta=1/3$, i.e., $\lambda=1$, when $2/s+3/q-2=1$. Moreover, as \eqref{eq: parameter3} implies that
\begin{equation*}
    \beta=\frac{1}{2}-\alpha \vartheta \bigg(\frac{3}{q}-\frac{5}{2}\bigg) - \alpha, \quad \gamma=\frac{1}{2}+\alpha \vartheta \bigg(\frac{3}{q}-\frac{5}{2}\bigg),
\end{equation*}
the setting $0 \leq \beta,\gamma \leq 1$ yields that
\begin{equation*}
    -\frac{1}{2} \leq \alpha \vartheta \bigg(\frac{3}{q}-\frac{5}{2}\bigg) \leq \frac{1}{2}-\alpha,
\end{equation*}
which is compatible with \eqref{eq: parameter4} only when $\alpha \vartheta \leq s/3$. Therefore, when $s=1$, we may also take $\alpha \vartheta=1/3$, i.e., $\lambda=1$. In other cases, it is legitimate to assume $\lambda<1$.

The rest proof is the same as Subsection \ref{subsec: (i)}, if we recall the simple fact that
\begin{equation*}
    \begin{split}
    C(f;r) \leq& c_0 C\big(f-[f]_{B_\rho};r\big)+c_0 C\big([f]_{B_\rho};r\big) \\
    \leq& c_0 \bigg(\frac{\rho}{r}\bigg)^2 C\big(f-[f]_{B_\rho};\rho\big) + c_0 \bigg(\frac{r}{\rho}\bigg) C(f;\rho).
    \end{split}
\end{equation*}

\subsection{Mixed Lebesgue norms of \texorpdfstring{$\curl u$ and $\curl b$}{curl u and curl b}} \hfill \\
\label{subsec: (iii)}

Recall the following result that has essentially been proved in \cite[Lemma 3.6]{gustafson2007Interior}.
\begin{lemma}
\label{thm: gradient<curl}
Suppose $\nabla f \in L_t^s L_x^q(Q_\rho)$ with $0 \leq 2/s+3/q-2 \leq 1$, $1 \leq s \leq \infty$, $(s,q) \neq (\infty,1)$ and $(s,q) \neq (1,\infty)$, then for any $0<2r<\rho$, it holds that
\begin{equation*}
    K(\nabla f,s,q;r) \leq c_0 \bigg(\frac{\rho}{r}\bigg)^{2/s+3/q-2} K(\curl f,s,q;\rho) + c_0 \bigg(\frac{r}{\rho}\bigg)^{2-2/s} K(\nabla f,s,q;\rho).
\end{equation*}
Furthermore, if $s=1$, $3 \leq q<\infty$, then
\begin{equation*}
    K(\nabla f,s,q;r) \leq c_0 \bigg(\frac{\rho}{r}\bigg)^{3/q} K(\curl f,s,q;\rho) + c_0 \bigg(\frac{r}{\rho}\bigg) K(\nabla f,s,q;\rho) + \mathcal{R}(f;r),
\end{equation*}
where $\mathcal{R}(f;r) \leq c_0 \rho^{-3/q} \int_{-r^2}^0 \|\nabla f(t)\|_{L^q(B_\rho)} \, \d t \to 0$ as $r \to 0$.
\end{lemma}
\begin{remark}
Similar argument also shows that if $\curl f \in L_t^s L_x^q$ near $(0,0)$, then so does $\nabla f$, where $s,q$ lie within the same range as Lemma \ref{thm: gradient<curl}.
\end{remark}

Starting from Lemma \ref{thm: gradient<curl}, we can immediately deduce from a standard iteration argument similar to the previous text that
\begin{equation*}
    \begin{split}
    K(\nabla f,s,q;r) \leq& c_1(\theta) \bigg(\frac{r}{r_0}\bigg)^\alpha K(\nabla f,s,q;r_0) + c_2(\theta) \sup_{0<r<r_0} K(\curl f,s,q;r) \\
    &+ c_3(\theta) \mathcal{R}(f;r_0), \quad \forall 0<r<\theta r_0
    \end{split}
\end{equation*}
for some $\alpha>0$, some small parameter $\theta>0$, and arbitrary small $r_0>0$. By letting first $r$ and then $r_0 \to 0$, we obtain $\varlimsup_{r \to 0} K(\nabla f,s,q;r) \leq c_2(\theta) \varlimsup_{r \to 0} K(\curl f,s,q;r)$. Therefore, when $(s,q)$ lies in the range in Lemma \ref{thm: gradient<curl}, the results of Theorem III and Theorem \ref{thm: boundedness_mixed} in terms of the curls follow from the results in terms of the gradients. The endpoint case $(s,q)=(1,\infty)$ can be derived from the cases where $s=1$, $3 \leq q<\infty$ and the H\"older inequality. In this way, we complete the proof of Theorem III and Theorem \ref{thm: boundedness_mixed}. \hfill$\Box$

\vskip 3mm
\noindent{\bf Conflict of interest.} {No potential conflict of interest was reported by the authors.}

\vskip 3mm
\noindent{\bf Ethics approval.} {Not applicable.}

\vskip 3mm
\noindent{\bf Funding.} {Hu was supported by China Scholarship Council, and Zhang was supported by NSFC grants (No. 12271423) and the Shaanxi Fundamental Science Research Project for Mathematics and Physics (No. 23JSY026).}

\vskip 3mm
\noindent{\bf Data availability.} {No data was used for the research described in the article.}

\end{CJK}
\end{document}